\documentclass{article}
\usepackage{amsmath}
\usepackage{amssymb}
\usepackage{amsthm}
\usepackage{array}
\usepackage[dvipsnames]{xcolor}
\usepackage{graphicx}
\usepackage{tikz}
\usepackage{tikz-cd}
\usetikzlibrary{matrix, calc, arrows}
\usepackage{xcolor}
\definecolor{mygray}{gray}{0.7}

\newtheorem{theorem}{Theorem}[section]
\newtheorem{lemma}[theorem]{Lemma}
\newtheorem{corollary}[theorem]{Corollary}

\title{Zeta functions of projective hypersurfaces with ordinary double points.}

\begin{document}

\author{Vladimir Baranovsky and Scott Stetson}

\maketitle

\abstract{
	We extend the approach Abbott, Kedlaya and Roe to computation of
	  the zeta function of a projective hypersurface 
	with $\tau$ isolated ordinary double points  over
   a finite field $\mathbb{F}_q$ 
	given by the reduction of a homogeneous polynomial $f \in \mathbb{Z}[x_0, \ldots, x_n]$, 
    under the assumption  of  equisingularity over $\mathbb{Z}_q$. The algorithm 
    is based on the results of Dimca and Saito (over the field $\mathbb{C}$ of complex 
    numbers) on the pole order spectral sequence in the 
case of ordinary double points. We give some examples of explicit computations for 
surfaces in $\mathbb{P}^3$.}

\section{Introduction}

Let $X_k = Z(f) \subset Y = \mathbb{P}^n_k$ be a projective hypersurface of degree 
$N$ over
a finite field $k = \mathbb{F}_q$ (for $q = p^a$ a prime power for $p > n-1$) 
with $\tau$ ordinary double points and assume that the 
defining equation $f$ admits an equisingular lift (see Section 2 for definitions) $f_{\mathcal{O}} 
\in \mathbb{Z}_q [x_0, \ldots, x_n]$ where $\mathcal{O} := \mathbb{Z}_q
\subset K := \mathbb{Q}_q$. For computational convenience we will actually work with 
polynomials over $\mathbb{Z}$ or its finite extension, although much of the theory 
holds over $\mathcal{O}$. 

The purpose of this paper is to generalize the algorithm originally due to 
Kedlaya that computes 
the zeta function of a smooth projective hypersurface, to the case of ordinary double 
points. This generalization is based on the results of Dimca and Saito on the spectral
sequence of the de Rham complex with deformed differential $d + df \wedge(\cdot)$, see
\cite{DiSaM1}, \cite{DiSaM2}, \cite{SaM}. 

We follow the exposition in \cite{Ge} with adjustments needed for our (mildly)
singular case. We switch to the study of the
open complement $U_k = \mathbb{P}^n_k \setminus X_k$ and show that the
equisingular deformation assumption allows to view $U_{\mathcal{O}} = \mathbb{P}^n_{\mathcal{O}}
\setminus X_{\mathcal{O}}$ as an open complement $Y \setminus Z$ where 
$Y$ is a smooth projective scheme over $Spec(\mathcal{O})$ and $Z$
is a normal crossings divisior with irreducible components smooth over $Spec(\mathcal{O})$.
In fact, in our case $Y$ is simply the blowup of $\mathbb{P}^n_{\mathcal{O}}$ at the 
singular points of $X_{\mathcal{O}}$. 

The presentation $U_\mathcal{O} = Y \setminus Z$ allows us to use a result
of Baldasarri-Chiarelotto, cf. \cite{BaCh}, that identifies the rigid cohomology of $U_k$, as 
a vector space over $K$, with the de Rham cohomology of $U_K$. Notice that
apriori the de Rham cohomology space does not carry a natural action induced
by the Frobenius map and we need to use the  isomorphism with rigid 
cohomology to transfer it from the latter. 

Fortunately, $U_k$, resp. $U_K$ are affine over $Spec(k)$, resp. $Spec(K)$ and 
both de Rham and rigid cohomology are computed via fairly explicit complexes
involving the ring of functions on $U_k$ and its dagger completion. In other words,
we use the Monsky-Weischnitzer model of rigid cohomology which is available
for the smooth affine complement. The above
isomorphism on cohomology is induced by the embedding of the usual functions
into completed functions.  Hence the Frobenius-induced operator on de Rham cohomology 
involves applying (a lift of) the Frobenius map to a usual differential form, obtaining 
a form with coefficients in the completion and then adjusting it by an exact form to 
get an equivalent element in the uncompleted de Rham complex. This procedure is 
covered in Section 5.

As in \cite{Ge}, computation involves the pole order filtration on differential forms
over $U_K$, except that for hypersurfaces with ordinary double points the
corresponding spectral sequence degenerates not at the $E_0$ page (as it happens 
for smooth hypersurfaces) but at the $E_1$ page. This has been proved by 
Dimca and Saito for hypersurfaces with ordinary double points over $\mathbb{C}$ (and
 later extended by Saito to the case of quasi-homogeneous isolated singularities),
 and we adjust their results to our needs. 
 
 Unlike in the smooth case, the cohomology of the Koszul
 differential $df \wedge (\cdot)$ can occur not just at the $\Omega^{n+1}$ 
 diagonal, but also at the $\Omega^{n}$ diagonal. 
 When the degree of polynomial coefficients is high enough, 
 both groups stabilize to vector spaces of
 dimension  $\tau$, the number  of ordinary double points. Moreover, the 
 differential on $E_1$ page of the spectral sequence (induced by the 
 de Rham differential) eventually becomes an
 isomorphism and at the $E_2$ page the spectral sequence degenerates. 
 This reduces an apriori infinite computation on $E_0$ and $E_1$ pages 
 to a computation involving only $n$ rows of the $E_1$ page. 
 
 \bigskip
 \noindent
 In Section 2 we give definitions related to projective hypersurfaces with 
 ordinary double points which are equisingular over $\mathbb{Z}_q$. 
 We show that blowing up the singular set gives a divisor with 
 smooth normal crossings.  In Section 3, after recalling basic definitions 
 and results we identify rigid of $X_k$ and de Rham cohomology of $X_K$
 using a theorem of Baldasarri and Chiarelotto. In Section 4 we 
 apply the results of Dimca and Saito on the de Rham cohomology on 
 the principal open set. In Section 5 we explain a modified version of 
 Kedlaya's algorithm which  allows to import the Frobenius action 
 to the de Rham cohomology, via its isomorphism with rigid cohomology 
 computed through Monsky-Washnitzer approach. In Section 6 we briefly summarize
 the steps of our algorithm. Section 7 has a detailed discussion of a few examples in 
 the $n=3$ case (projective surfaces with ordinary double points). We tried to be 
 explicit in some details for the benefit of a reader interested in practical implementation 
 of our algorithm. 
 
 \bigskip
 \noindent 
 \textbf{Acknowledgements.} We are grateful to Prof. K. Kedlaya for sharing 
 the code of his algorithm. 
 
 \section{Equisingular deformation and its blowup.}
 
 \subsection{Definition of an equisingular ODP polynomial.}
 \label{equi}
 
 Assume we are given a homogeneous degree $N$ polynomial 
 $f \in \mathbb{Z}[x_0, \ldots, x_n]$ (one can also replace $\mathbb{Z}$ by the 
 ring of integers in a finite extension of $\mathbb{Q}$). For every commutative ring $R$
 it defines a relative hypersurface $X_R$  in the projective space $\mathbb{P}^n_R$ over
 $Spec(R)$. We will be interested in the special cases when $R$ is a finite field
 $k = \mathbb{F}_q$  with $q$ elements where $q=p^a$ for some prime $p > n-1$, 
 or unramified degree $a$ extension $K = \mathbb{Q}_q$ of the field $\mathbb{Q}_p$ of 
 $p$-adic numbers, or the ring $\mathcal{O} = 
 \mathbb{Z}_q$ of integers over $\mathbb{Z}_p$ in $K$. 
 In the case when  $R = k, K$ is a field, we have the standard definitions:
 
 \bigskip
 \noindent
 \textbf{Definition:} A point $P\in Z(f)\subset\mathbb{P}^n_F$ is a 
 \textit{singular point of 
 the hypersurface} $Z(f)$  if
 \begin{equation*}
 	\frac{\partial f}{\partial x_0}(P)=\dots=\frac{\partial f}{\partial x_n}(P)=0.
 \end{equation*}
Such a $P$ is an \textit{ordinary double point} (ODP) if the homogeneous $(n+1) \times 
(n+1)$  
Hessian matrix $
 	\Bigg(\frac{\partial^2f}{\partial x_i\partial x_j}\Bigg)_{0 \leq i, j \leq n}$
 has rank $n$ at $P$. This is equivalent to requiring that $P$ has a standard affine neighborhood
$U \simeq \mathbb{A}^n_R \subset \mathbb{P}^n_R$ such that the non-homogeneous
polynomial defining $Z(f) \cap U$ has a non-denerate affine $n \times n$ Hessian matrix 
at $P$.

\bigskip
\noindent 
Yet another way to phrase it is to use the partial derivatives $\frac{\partial f}{\partial x_i}$
to define a morphism 
\begin{equation}
\label{coker}
\mathcal{O}_{\mathbb{P}^n_R} (-N+1)^{\oplus (n+1)}\to 
\mathcal{O}_{\mathbb{P}^n_R} 
\end{equation}
and the ODP condition is equivalent to saying that for $R = F$ the cokernel of this morphism 
is supported on a zero dimensional \textit{reduced} subscheme of $\mathbb{P}^n_F$. 
 In the case when $R = \mathcal{O}$, we need a condition that ensure that 
 the ordinary double points of the induced hypersurface over 
 $k = \mathcal{O}/ \mathfrak{m}$ are not ``smoothed away" on the hypersurface
 over $K$, the fraction field of $\mathcal{O}$.

 \bigskip
 \noindent
 \textbf{Definition:}  A hypersurface $Z(f) \subset \mathbb{P}^n_{\mathcal{O}}$ given 
 by a homogeneous degree $N$ polynomial $f  \in \mathbb{Z}[x_0, \ldots, x_n]$
 has \textit{equisingular ordinary double points} if the cokernel of the morphism 
 (\ref{coker}) is scheme theoretically supported at a  disjoint union $W = \coprod W_s$ of closed subschemes $W_s$, 
 $s = 1, \ldots, \tau$
 such that the restriction of the structure morphism $\pi: \mathbb{P}^n_{\mathcal{O}}
 \to Spec(\mathcal{O})$ to each $W_\alpha$ is an isomorphism.

  \bigskip
  \noindent 
  In this paper we assume that $p > n - 1 $ (this is needed to apply 
  truncation results in Section o4 of \cite{Ge}, but not essential elsewhere) 
  and that $f  \in \mathbb{Z}[x_0, \ldots, x_n]$  is
  a homogeneous degree $N$ polynomial which induces a hypersurface with $\tau$
  equisingular ordinary double points over $Spec(\mathcal{O})$ (or a polynomial 
  with coefficients in a finite extension of $\mathbb{Z}$ which is unramified at $p$).  
 
 \subsection{Blowup of the equisingular hypersurface.}
 
We continue with the assumptions on $f$ imposed in the previous subsection. 
 
 \begin{lemma}  
 	\label{blowup}
 	Let $\rho: Y \to \mathbb{P}^n_{\mathcal{O}}$ be the blowup of the 
 	closed subscheme $W$ introduced at the end of the previous subsection. 
 	Then the preimage $D$ of $X = Z(f)$ (with the reduced scheme structure):
 	$$
 	D = \rho^{-1} (X) = \widehat{X} \cup E_1 \cup \ldots \cup E_\tau
 	$$
 	has smooth irreducible components with normal crossings. 
 	\end{lemma} 

 \begin{proof}
 	Note that the question is local along $D$ and we can restrict to 
 	a neighborhood of a point $Q \in D$. 
 	
 	If $\rho(Q)$ is not in the
 	support of $W$ then the only component of $D$ that passes through 
 	$Q$ is the proper transform $\widehat{X}$ which is smooth at $Q$
 	since $\rho(Q)$ is a smooth point of $X$ and $\rho$ is an isomorphism 
 	over some affine neighborhood of $\rho(Q)$. 
 	
 	Now assume $P = \rho(Q)$ is in the support of some $W_s$. We can assume 
 	$s = 1$ are replacing $Q$ by its specialization we can assume that $Q$ and
 	$P$ are in the respective closed fibers over $Spec(k) \subset Spec(\mathcal{O})$. 
 	We can also assume that $P = [1: 0: \ldots : 0] \in \mathbb{P}^n_k$. We can replace 
 	the projective space by $\mathbb{A}^n_{\mathcal{O}}$ with affine coordinates
 	$y_1 = x_1/x_0, \ldots, y_n = x_n/x_0$ and $X$ by the zero set of $g (y_1, \ldots, 
 	y_n) = f(1, y_1, \ldots, y_n)$. 
 	
 	By the Euler identity $g$ and its partials generate the ideal $I$ of the subscheme
 	$W_1$. By assumption on $W_1$ this means that the constant and linear terms
 	of $g$  are zero, hence we can write $g = g_2 + g_3 + \ldots + g_N$ where 
 	each $g_j$ is a homogeneous degree $j$ polynomial in $y_1, \ldots, y_n$ with 
 	coefficients in $\mathcal{O}$.
 	
 	By the standard results on blowups, see e.g. Chapter 1.4 of \cite{Ha}, the 
 	exceptional divisor $E_1$ of the blowup of $W_1$ in $\mathbb{A}^n_{\mathcal{O}}$ is 
 	smooth. The proper transform of $X$ is smooth away from the exceptional divisor
 	so we just need to show that at every point $Q \in \widehat{X} \cap E_1$ the 
 	proper transform is smooth at $Q$ and the two hypersurfaces are transversal 
 	(i.e. the differetials over $\mathcal{O}$ of the two local definiing equations are
 	independent). Both facts hold since $g$ and its partial derivatives generate $I$ 
 	hence the linear parts of partial derivatives freely generate the $\mathcal{O}$
 	of linear polynomials in $y_1, \ldots, y_n$ with coefficients in $\mathcal{O}$. 
 \end{proof}
 
 \subsection{A non-equisingular example.}
 For an example of a lift that is not equisingular, for $n=3$, we use Example 5.6 on page 22 of \cite{Ve} which is a quartic with 2 ordinary double points over $\mathbb{C}$ but 4 over
 the closure of $\mathbb{F}_5$.  This surface is also mentioned on page 7 of \cite{DeA} where the variables $x_1$ and $x_3$ have been swapped and then written in a ``cleaner" form
 \begin{equation} \label{eq:quartic2nodes}
 	x_0x_1(x_0^2+x_1^2+x_2^2+x_3^2)+x_2x_3(x_0^2+x_1^2-x_2^2-x_3^2)-2x_2^2x_3^2+2x_0^2x_1^2+2x_0x_1x_2x_3=0.
 \end{equation}
Over a characteristic zero field, the polynomial $f$ on the left hand side has two ordinary double points  at $[1:-1:0:0]$ and $[0:0:-1:1]$.  Consider it now over a field of characteristic 5. One must find the number of singularities over $\mathbb{F}_5$ and its finite extensions.     Let $\alpha\in\mathbb{F}_{25}$ such that $\alpha^2=3$, then the singular points of the quartic in (\ref{eq:quartic2nodes}) are
 \begin{align*}
 	[1:4:0:0] && [0:0:4:1] && [\alpha:\alpha:1:1] && [4\alpha:4\alpha:1:1].
 \end{align*}
We conclude that our hypersurface is not equisingular in this case.
 
 \section{Zeta function, rigid and de Rham cohomology}
 
 \subsection{Zeta function and traces on rigid cohomology}
 	We recall the standard definitions and results regarding zeta functions, as they 
 	apply to the situation considered.  The projective zeta function of $f$ is defined as
 	\begin{equation} \label{eq:zetaDef}
 		\zeta (f,T)=\exp\Bigg(\sum_{r=1}^{\infty}\#X_{k(r)} \frac{T^r}{r}\Bigg)
 	\end{equation}
 	where $X_{k(r)}=\{x\in\mathbb{P}^n_{k(r)}|f(x)=0\}$ and  $k(r)$ stand for the finite feld
 $\mathbb{F}_{q^r}$. 
 	
 	In \cite{Dw} Dwork proved that the zeta function of a variety is a rational function. Weil had proposed using some type of cohomology theory to prove his conjectures however Dwork's proof uses $p$-adic analysis and no cohomology. It wasn't until Grothendieck's use of étale cohomology \cite{Gr} that the rationality of the zeta function was proved using cohomological methods. This breakthrough led to the proofs of all the Weil conjectures, the last being the analogue of the Riemann Hypothesis by Deligne \cite{DeP}. When it comes to actually computing the zeta function, Kedlaya in section 1.2 of \cite{Ke} points out that, étale cohomology is not as useful as $p$-adic cohomology such as rigid, see also the introduction of \cite{Ge}.  Rigid cohomology can
 	defined both affine and projective varieties whether they be smooth or singular, and in fact there are
 	two versions: usual and with compact support. For our projective hypersurface $X$ constructed
 	from $f$ we write $H^\bullet_{rig}(X)$, $H^\bullet_{rig, c}(X)$, both finite dimentional 
 	vector spaces over the field $K$. It is an important part of the 
 	standard theory, cf. \cite{LS}, that the Frobenius automorphism of $k$ induces a linear map 
 	$F^*$ on both version of cohomology, such that the following crucial trace formula holds
 	\begin{equation} \label{eq:trace}
 		\#X_{k(r)}=\sum_{i=0}^{2\dim(X)}(-1)^i\text{tr}((F^*)^r|H^i_{\text{rig,c}}(X))
 	\end{equation}
 	Substituting equation (\ref{eq:trace}) into the definition of the zeta function yields
 	\begin{equation} \label{eq:zetaProd}
 		\zeta(f,T)=\prod_{i=0}^{2\dim(X)}\det(1-TF^*|H^i_{\text{rig,c}}(X))^{(-1)^{i+1}}.
 	\end{equation}
 	where $\det(1-TF^*|H^i_{\text{rig,c}}(X))$ is the reciprocal of the characteristic polynomial of the map $F^*$. Equivalently, in the smooth case by Poincare duality,  see Proposition 6.4.18 in \cite{LS}, the information contained on the compactly supported rigid cohomology can
 	be rephrased in terms usual rigid cohomology: 
 	\begin{equation} \label{eq:trace-dual}
 		\#X_{k(r)}=\sum_{i=0}^{2\dim(X)}(-1)^iq^{r(n-1)}\text{tr}((F^*)^{-r}|H^i_{\text{rig}}(X))
 	\end{equation}
 	
 	\bigskip
 	\noindent 
 	We now consider the complement $U_k=\mathbb{P}^n_k\setminus X_k$. 
 	 This section and the following are driven by the trivial, but important fact that over $\mathbb{F}_q$
 	\begin{equation*}
 		|X_k|+|U_k|=|\mathbb{P}^n_k|=q^n + \ldots +q^2+q+1.
 	\end{equation*}
 	Therefore finding the zeta function of $X$ is equivalent to finding the zeta function of $U$ and we choose to calculate the latter.  The two main reasons for working with $U$, instead of $X$, are that $U$ is smooth and affine. Yet another feature which makes our particular approach computable
 	is the fact that we choose $f$ giving a hypersurface with equisingular ordinary double points. 
 
 	Using the above equation we can write the zeta function of $f$ as
 	\begin{equation}
 		\label{eq:defineP}
 		\zeta(f,T)=\frac{Q_n(T)^{-n} Q_{n+1}(T)^{-n-1}}{(1-T)(1-qT)\ldots (1-q^nT)} 
 	\qquad 
 		Q_i(T)=\det(1-q^iT(F^*)^{-1}|H^i_{\text{rig,c}}(U_k)).
 	\end{equation}
 (seee e.g. \cite{Ge} or \cite{LS} ).
 	In Section 4 we use results on spectral sequences to prove that
 	\begin{equation} \label{eq:degP}
 		\deg(Q_n(T)) - \deg(Q_{n+1}(T)) =\frac{1}{N}((N-1)^{n+1}+N-1)-\tau.
 	\end{equation}
 Moreover, for $n$ odd $\deg(Q_{n+1}(T))=0$ while for $n$ even $\deg(Q_{n+1}(T)) = \tau$. 
 
 \subsection{Isomorphism with de Rham cohomology}
 
Lemma \ref{blowup} has an important consequence, see Theorem 5.2 in \cite{ChSt} 
 and Corollary 2.6 in \cite{BaCh} 
 
 \begin{theorem}
 	\label{thm:isom}
 	In the situation considered in section \ref{equi}, 
 	there exists an isomorphism between the algebraic
 	de Rham cohomology of $U_K$ and the rigid cohomology of $U_k$:
 	$$
 	H^\bullet_{DR}(U_K) \to H^\bullet_{rig}(U_k) 
 	$$
 	\end{theorem}
An important detail is that apriori the algebraic  de Rham cohomology group does not carry 
an action of an operator induced by Frobenius, hence some information on the above 
isomorphism is needed to transport this action from rigid cohomology, in terms which 
are explicit enough for the computaiton. 

At this point we use the fact that $U$ is smooth and affine: the de Rham cohomology can 
be computed in terms of the algebra $A$ of regular functions on $U_K$ while for the
rigid cohomology we can use the Monsky-Washnitzer model, which is rather close to 
de Rham but uses a certain ``dagger completion" $A^\dagger$. Thus, an important property of 
the isomorphism quoted above is that in the smooth affine case it is induced by 
the embedding $A \to A^\dagger$.  This follows from the proofs in \cite{BaCh} and \cite{ChSt}. 
We delay the detailed discussion of this point until Section 5.

 \section{Pole order spectral sequence and degeneration}
 
 For most of this Section we will discuss results of Dimca and Saito which were originally 
 stated over $\mathbb{C}$, eventually using an embedding $K \to \mathbb{C}$ to relate
 these results to our situation. 
 
 \subsection{de Rham cohomology of the principal open set}
 Recall that $U = \mathbb{P}^n \setminus Z(f)$ is both smooth and affine so the computation of
 the de Rham cohomology 
 does not need to involve a Čech covering.  We follow Chapter 6 of Dimca's book \cite{Di}.  Let $S=\mathbb{C}[x_0,x_1,\ldots,x_n]$ and consider the graded $S$-module $\Omega^l$ of polynomial differential $k$-forms.  Any element of $\Omega^l$ can be written as
 \begin{equation*}
 	\omega=\sum_{I}c_Idx_{i_1}\wedge\dots\wedge dx_{i_l}
 \end{equation*}
 where the finite summation runs over $I=(0 \leq i_1 < \dots < i_l \leq n)$ and $c_I\in S$.  
  In this paper we only consider the trivial weights, $\deg(x_i)=\deg(dx_i)=1$ making the total degree
 \begin{equation*}
 	\deg(x_0^{a_0}\dots x_n^{a_n}dx_{i_1}\wedge\dots\wedge dx_{i_l})=a_0+\dots+a_n+l.
 \end{equation*}
 Now let $S_m\subset S$ be the set of homogeneous polynomials of degree $m$ and define $\Omega^l_m$ to be the set of $l$-forms whose total degree is $m$, i.e., if $\omega\in\Omega^l_m$ then using the notation of (1.12) each $c_I$ is in $S_{m-l}$.
 
 Relating  polynomial differential forms to forms on $U_K$ is greatly simplified by using the
  map $\Delta$ which is the contraction with the Euler vector field $E=\sum_{i=0}^{n}x_i\partial/\partial x_i$.  By Lemma 1.15 from \cite{Di} 
 \begin{lemma}
  (A) There is a unique $S$-linear operator $\Delta:\Omega^l_m\rightarrow\Omega^{l-1}_m$ satisfying the properties:
 \newline (i) $\Delta(\omega\wedge\omega')=\Delta(\omega)\wedge\omega'+(-1)^l\omega\wedge\Delta(\omega')$ for $\omega\in\Omega^l, \omega'\in\Omega^{s}$;
 \newline (ii) $\Delta(df)=Nf$ for any $f\in S_N$.
 \newline (B) For a homogeneous differential form $\omega\in\Omega^l_m$ we have
 \begin{equation*}
 	\Delta d(\omega)+d\Delta(\omega)= m \omega.
 \end{equation*}
 \newline (C) The sequence 
 \begin{equation*}
 	0\rightarrow\Omega^{n+1}\xrightarrow{\Delta}\Omega^n\xrightarrow{\Delta}\dots\xrightarrow{\Delta}\Omega^1\xrightarrow{\Delta}\Omega^0\rightarrow0
 \end{equation*}
 is exact, except for the last term where
 \begin{equation*}
 	\text{im}(\Delta:\Omega^1\rightarrow\Omega^0)=(x_0,\dots,x_n),
 \end{equation*}
 the maximal ideal in $S$ generated by $x_0,\dots,x_n$. $\square$
 \end{lemma}
 
 \bigskip
 \noindent
 Note that property (ii) is a restatement of Euler's Identity in terms of $\Delta$ and $d$.  Now by Proposition 1.16 of \cite{Di}, any differential $l$-form $\omega$ (for $l>0$) on the open set $U$ can be written as
 \begin{equation} \label{eq:formOnU}
 	\omega=\frac{\Delta(\gamma)}{f^s}
 \end{equation}
 for some integer $s>0$ and $\gamma\in\Omega^{l+1}_{sN}$.  At this point Dimca asks how can $d\omega$, from (\ref{eq:formOnU}), be expressed in a similar form  
 \begin{equation} \label{eq:domega}
 	d\omega=\frac{\Delta(\delta)}{f^{s+1}}
 \end{equation}
 for some $\delta\in\Omega^{l+2}_{(s+1)N}$.  Most calculations below can be found on page 181 of \cite{Di} but we include them since the formula for  $\delta$ in (\ref{eq:domega}) gives rise to a double complex and the spectral sequence, on which our algorithm is based.  This is covered in the next section,  while now we consider $d\omega$ from (\ref{eq:domega})
 \begin{equation}
 	d\omega=d\Bigg(\frac{\Delta(\gamma)}{f^s}\Bigg)=d(f^{-s}\wedge\Delta(\gamma))
 	 =-sf^{-s-1}df\wedge\Delta(\gamma)+f^{-s}\wedge d(\Delta(\gamma)). \label{eq:toSimplify}
 \end{equation}
 By Lemma 4.1 (A), (B)
 \begin{equation}
 	\Delta(df\wedge\gamma)   =Nf\gamma-df\wedge\Delta(\gamma); 
 	\qquad
 	d\Delta(\gamma)=sN\gamma-\Delta d(\gamma), \text{ for }\gamma\in\Omega^{l+1}_{sN}. \label{eq:part2}
 \end{equation}
 Substituting (\ref{eq:part2}) into (\ref{eq:toSimplify}) yields
 \begin{equation*}
 	d\omega  =-sf^{-s-1}(Nf\gamma-\Delta(df\wedge\gamma))+f^{-s}\wedge (sN\gamma-\Delta(d\gamma)) =-\frac{\Delta(fd\gamma-sdf\wedge \gamma)}{f^{s+1}}
 \end{equation*}
 Finally we can write $d\omega$, from (\ref{eq:domega}), as
 \begin{equation} \label{eq:domegaFull}
 	d\omega=d\Bigg(\frac{\Delta(\gamma)}{f^s}\Bigg)=-\frac{\Delta(d_f(\gamma))}{f^{s+1}}
 	\qquad \textrm{ where } d_f(\gamma) := f d\gamma - \frac{|\gamma|}{N} df \wedge \gamma. 
 \end{equation}
In other words,
 for known $f^s$ in the denominator, computation of the de Rham differential of a form on $U$ 
reduces to computation of a ``deformed" differential of a form on $\mathbb{C}^{n+1}$, built out
of $d$ and $df\wedge$.  These differentials give rise to a double complex considered below.

The total degree of a polynomial differential form is preserved both by the
de Rham differential and the Euler vector field contraction. Since differential forms on 
$U$ are represented by fractions of homogeneous degree zero, we can restrict to polynomial differential forms of total degree $sN$ for $s \geq 0$. 
 
 \subsection{Pole order spectral sequence.}

 The formulas of the previous section allow us to obtain information about the de Rham cohomology of
 $U_{\mathbb{C}}$. On one hand, we have a dg algebra $\Omega(A)^\bullet$ given by the global sections
 of the de Rham complex with the filtrations
 $$
 F^s \Omega(A)^l = \Big\{ \frac{\Delta(\gamma)}{f^a}  \textrm{ such that } 
 \gamma \in \Omega^{l+1}_{aN}, |\gamma| = a N \textrm{ and } a \leq l - s\Big\}
 $$
 (which is compatible with the de Rham differential by the previous subsection). On the other hand
 we have a bicomplex $B^{t, s}  = \Omega^{s+t+1}_{sN}$ with differentials $d' = d$ and 
 $d''(\gamma) = - \frac{|\gamma|}{N} df \wedge \gamma$ and the filtration 
$$
F^s B^l = \bigoplus_{r \geq s} B^{r, l-r}
$$
The two filtrations give standard spectral sequences $E_r^{t, s} (A)$ and $E_r^{t, s} (B)$ with 
$E_1^{t, s} $ terms given by $H^{s+t}(F^s(\cdot)/F^{s+1}(\cdot))$ and a morphism $\{\delta_r\}_r$
	of spectral sequences induced by $\gamma \mapsto \frac{\Delta(\gamma)}{f^t}$
	for $\gamma \in B^{t, s} $.  To formulate the next result one needs to adjust both spectral
	sequences by replacing $E_0^{0, 0}(A) = \mathbb{C}$ by zero and mod out 
	$E_0^{-1, 0}(B)$, $E_0^{0, 0}$ by the 1-dimensional subspaces spanned by $1$ and $df$, respectively. 
	After this modification (which kills the de Rham cohomolgy of $U$ in degree zero and
	does not affect $E_1(B)$) we can state Theorem 6.2.9 in \cite{Di}:
	\begin{theorem}
		The morphism of reduced spectral sequences $\{\delta_r\}: \widetilde{E}_r^{t, s} (A) \to 
		\widetilde{E}_r^{t, s} (B)$ is an isomorphism for all $r \geq 1$. 
	\end{theorem}

\noindent
	In fact, we will see below that in our case all differentials on pages $E_r$ with $r \geq 2$ are 
	trivial.

 \bigskip
 \noindent
  To focus on $E_0(B)$, as shown in  Figure \ref{fig:E0page}, recall that $\Omega^l_{j+l} =
  S_j\Omega^l$
 be the set of all $l$-forms whose coefficients are homogeneous polynomials is $S=\mathbb{C}[x_0,x_1,\ldots,x_n]$ of degree $j$.  
  For each module $S_j\Omega^k$ in Figure \ref{fig:E0page},  $j+k$ is a multiple of $N=\deg(f)$.  
   We set $S_{N-k}\Omega^k=0$ whenever the degree of $N$ is less than $k$.
   The vertical arrows in Figure \ref{fig:E0page} are the Koszul differentials $(-s) \cdot df\wedge$ (the exterior product) and the horizontal arrows are grayed out because we apply them on the $E_1$ page, are induced the de Rham differentials $d$.  The two differentials  $df\wedge$ and  $d$ anti-commute.

 \begin{figure}[h!]
 	\begin{center}
 		\begin{tikzpicture}
 			\matrix (m) [matrix of math nodes,row sep=2.8em,column sep=2em,minimum width=2em]
 			{
 			    \ddots  & \ddots & 0 & & & \\
 				 S_{3N-n+1}\Omega^{n-1} & S_{3N-n}\Omega^n & S_{3N-n-1}\Omega^{n+1} & 0 & \\
 				 \dots & S_{2N-n+1}\Omega^{n-1} & S_{2N-n}\Omega^n & S_{2N-n-1}\Omega^{n+1} & 0  \\
 				 & \dots & S_{N-n+1}\Omega^{n-1} & S_{N-n}\Omega^n & S_{N-n-1}\Omega^{n+1}  \\};
 			\path[-stealth]
 			(m-2-1) edge (m-1-1)
 			(m-2-2) edge (m-1-2)
 			(m-2-3) edge (m-1-3)
 			(m-3-2) edge (m-2-2)
 			(m-3-3) edge (m-2-3)
 			(m-3-4) edge (m-2-4)
 			(m-4-3) edge (m-3-3)
 			(m-4-4) edge (m-3-4)
 			(m-4-5) edge (m-3-5)
 			
 			(m-2-1.east|-m-2-2) edge [mygray] (m-2-2)
 			(m-2-2) edge [mygray] (m-2-3)
 			(m-2-3) edge [mygray] (m-2-4)
 			(m-3-1.east|-m-3-2) edge [mygray] (m-3-2)
 			(m-3-2) edge [mygray] (m-3-3)
 			(m-3-3) edge [mygray] (m-3-4)
 			(m-3-4) edge [mygray] (m-3-5)
 			(m-4-2.east|-m-4-3) edge [mygray] (m-4-3)
 			(m-4-3) edge [mygray] (m-4-4)
 			(m-4-4) edge [mygray] (m-4-5);
 		\end{tikzpicture}
 	\end{center}
 	\caption{$E_0$ page of the Spectral Sequence.}
 	\label{fig:E0page}
 \end{figure}
 The terms on page $E_1(B)$ are the appropriate homogeneous components of the 
 Koszul cohomology, i.e. cohomology of the complex  $K^\bullet_f$
 \begin{equation*}
 	0\xrightarrow{df\wedge}\Omega^0\xrightarrow{df\wedge}\Omega^1\xrightarrow{df\wedge}\Omega^2\xrightarrow{df\wedge}\ldots\xrightarrow{df\wedge}\Omega^{n+1}\xrightarrow{df\wedge}0.
 \end{equation*}
 It is well known that if $f$ has isolated singularities then Koszul cohomology groups are trivial except for the top two, i.e., $H^i(K^{\bullet}_f)=0$ for $i<n$, see page 2 of \cite{SaM} and Proposition 
 6.2.21 on page 195 of \cite{Di} for the statement and \cite{SaK} for a proof.   Let
 \begin{align*}
 	H^n(K_f^{\bullet})_j & =\ker(S_j\Omega^n\xrightarrow{df\wedge}S_{j+N-1}\Omega^{n+1})/\text{im}(S_{j-N+1}\Omega^{n-1}\xrightarrow{df\wedge}S_j\Omega^n) \\
 	H^{n+1}(K_f^{\bullet})_j & =S_j\Omega^{n+1}/\text{im}(S_{j-N+1}\Omega^n\xrightarrow{df\wedge}S_j\Omega^{n+1})
 \end{align*}
 then the $E_1$ page of the spectral sequence is shown in Figure \ref{fig:E1page}.
 
 \begin{figure}[h]
 	\begin{center}
 		\begin{tikzpicture}
 			\matrix (m) [matrix of math nodes,row sep=1.5em,column sep=1.8em,minimum width=2em]
 			{
 				\ddots & \ddots & & & \\
 				& H^n(K^{\bullet}_f)_{3N-n} & H^{n+1}(K^{\bullet}_f)_{3N-n-1} & & \\
 				& & H^n(K^{\bullet}_f)_{2N-n} & H^{n+1}(K^{\bullet}_f)_{2N-n-1} & \\
 				& & & H^n(K^{\bullet}_f)_{N-n} & H^{n+1}(K^{\bullet}_f)_{N-n-1} \\};
 			\path[-stealth]
 			(m-2-2.east|-m-2-3) edge node [above] {$d$} (m-2-3)
 			(m-3-3.east|-m-3-4) edge node [above] {$d$} (m-3-4)
 			(m-4-4.east|-m-4-5) edge node [above] {$d$} (m-4-5);
 		\end{tikzpicture}
 	\end{center}
 	\caption{$E_1$ page of the Spectral Sequence.}
 	\label{fig:E1page}
 \end{figure}

 The $E_2$ page is defined by taking the kernel and cokernel of the de Rham differential on the spaces seen above.  (Technically the maps above are only induced by the de Rham differential since the elements of $H^i(K^{\bullet}_f)$ are equivalence classes of forms, but we abuse terminology).  
 
 \subsection{Results of Dimca and Saito.}

 Rather powerful results of Dimca and Saito, which we review below, claim that 
 the differential of the $E_1$ page is injecive in every position (with a single exception for even $n$) and is also surjective except perhaps
 in the first $n$ terms. In addition, the spectral sequence degenerates at the $E_2$ term. 
 \begin{theorem}
 	\label{e2-theorem} 
 	Denote by 
 	$H^n(K^{\bullet}_f)_{jN-n}\xrightarrow{d}H^{n+1}(K^{\bullet}_f)_{jN-n-1}$ the horizontal differential for the $E_1$ page of the spectral sequence (induced
 	by the de Rham differential). Then 
 	
 	(i) the spectral sequence degenerates at the $E_2$ page;
 	
 	(ii) the $E_1^{t, s} $ terms have dimension $\tau$ (i.e. the number of nodes)
 	when $s \geq (n+1)$ and 
 	$s+t = n$ or $s+t = n-1$.
 	
 	(ii) the vector spaces $E_2^{t, s} $ vanish except possibly in two cases: either when $ s \in \{1, \ldots, n \}$
 	and $t = n-s$, or when $n$ is even and the space $E_2^{n/2-1, n/2}$ has dimension $\tau$. 
 \end{theorem} 
\begin{proof}
	The spectral sequence degenerates at the $E_2$ page by Theorem 2 in \cite{SaM}. 
	We have seen above that $E_1^{t, s} $ can only be nonzero for $s+t = n-1$ or $n$.

	First consider the case $s+t = n-1$. Then 
	$\dim E_1^{t, s}  \leq \tau$ for all $s$ and the equality holds 
  for $s \geq n$ by Theorem 1 in \cite{Di1}. 
  On the other hand, this vector space vanishes
  for $s < \frac{n}{2}$ by Theorem 9 in \cite{DiSt} and the remark right after it. 
  In addition $\dim E_2^{t, s}  = 0$ unless $s = \frac{n}{2}$ (which can happen only when $n$ is
  even) and in the latter case the dimension is $\tau$. Both facts are proved in Theorem 5.3 
  of \cite{DiSaM2}. We note here that our $s$ equals $\frac{p}{d}$ in the notation of \textit{loc.cit.} 
  and that in the case of ordinary double points each expression $\alpha_{h_k, k}$ 
  of \textit{loc. cit.} is $\frac{n}{2}$ hence the count in the theorem quoted simply reduces
  to the number of ordinary double points.
  
  Now we turn to the case of $s+t = n$. Our goal is to show that for $s \geq (n+1)$ the 
  dimension of $\dim E_1^{t, s} $ is also equal to the number of singular points while 
  $E_2^{t, s} $ also vanishes in that range. Indeed, for $s \geq (n+1)$ and 
  any \textit{smooth} hypersurface of the same degree as $f$, the group 
  similar to $E_1^{t, s} $ is zero by Section 6.1  in \cite{Di}. Since the Euler characteristic
  of the complex $E_0^{s, *}$ with finite dimensional components and Koszul differential 
  $df$ is independent on the choice of $f$, we conclude that $\dim E_1^{s, n -s} 
  = \dim E_1^{s, n - s - 1} = \tau$ for $s \geq {n+1}$. 
  
  The $E_1$ page differential $d: E_1^{n - s, s-1} \to E_1^{n-s, s}$ is a linear map 
  of vector spaces of the same dimension, and its kernel (i.e. $E_2^{n-s, s-1}$)
  vanishes since $(s-1) \geq n  > n/2$. So the cokernel vanishes as well. 
  Finally, for $s =0$ and $t = n-1, n$ the groups $E_1^{0, t}$ vanish since a
  polynomial $l$-form with $l > 0$ has degree $\geq l$. 
\end{proof}
 
 \begin{corollary} Let 
 	$$
 	b(n, N) = \frac{1}{N} ( (N-1)^{n+1} + (-1)^{n+1} (N-1))
 	$$
 	be the dimension of $H^n_{dR}(U_K)$ for a \textrm{smooth} degree $N$ 
  hypersurface in $\mathbb{P}^n_K$, cf. Theorem 3.1 in \cite{Ge}, and $\mu$ 
  be the number of ordinary double points for $f$. Then 
  
  (i) If $n$ is odd then $H^n_{dR}(U_K)$ has dimension $b(n, N) - \tau$ and other
  reduced de Rham cohomology groups are zero 
  
  (ii) If $n$ is even then $H^n_{dR}(U, K)$ has dimension $b(n, N)$, 
  $H^{n-1}_{dR} (U_K)$ has dimension $\tau$, and other reduced de Rham cohomology 
  groups are zero. 
 	
 \end{corollary}

\begin{proof}
	The previous theorem (degeneration of pole order spectral sequence plus computation in 
	part (ii)) establish the assertion for $H^{n-1}$.
	
	For $H^n$, let's compute $\dim H^{n}_{dR} (U_K) - \dim H^{n-1}_{dR}(U_K)$.
	By degeneration, this is equal to 
	$$
	\sum_{s = 0, \ldots, n} (\dim E_2^{n - s, s} - \dim E_2^{n-s-1, s})
	= \sum_{s = 0, \ldots, n} (\dim E_1^{n - s, s} - \dim E_1^{n-s-1, s} ) - \tau 
	= b(n, N) - \tau
	$$
	and the assertion follows. The first equality above holds since $E_2^{n - s - 1, s}$ 
	vanishes for $s = n+1$ and $E_1^{n - s - 1, s}$ has dimension $\tau$ 
	by Theorem 1 in \cite{DiSt}, as quoted before. So the difference of 
	dimensions for $E_2^{n - s, s}$ and $E_1^{n-s, s}$ for $s = n$ is precisely $\tau$. 
	As for all other values of dimensions, we can related $E_1$
	page to $E_2$ by using the Euler characteristic property for the $E_1$ page 
	differential. Finally, the second equality follows similarly by the Euler characteristic
	of the $E_0$ page differential: since it is independent on the choice of differential 
	we can replace $f$ by a homogeenous polynomial defining a smooth hypersurface 
	where the standard techniques given $b(n, N)$. 
\end{proof}

\bigskip
\noindent 
\textbf{Remark.} Observe that although all results of Dimca and Saito 
were obtained over the field $\mathbb{C}$, they immediately transfer to our choice
of $K$. Indeed, there exists a (discontinuous) embedding $K \to \mathbb{C}$, 
all our complexes have filtrations with finite dimensional quotients
hence all statements regarding vanishing and dimension over $K$ are equivalent to 
the complex valued version by exactness of $\otimes_K \mathbb{C}$. 
 
 \section{Frobenius action}
 
 \subsection{Dagger algebra and Monsky-Washnitzer model of rigid cohomology}
 
 Next we compute the Frobenius action induced by the embedding, see Section 2 of 
 \cite{Ge},  of $A\hookrightarrow A^{\dagger}$ where $A$ is the ring of functions on $U=\mathbb{P}^3\setminus Z(f)$ and $A^{\dagger}$ is a certain completion used in Monsky-Washnitzer cohomology theory.  We will not repeat the definition here, sending the
 reader to \textit{loc. cit.}, but infinite sums used below are well defined in
 $A^\dagger$. The important reason for introducing $A^\dagger$ is that the corresponding completed
 de Rham complex $(\Omega^\bullet(A^{\dagger}), d)$ computes the rigid cohomology 
$H^\bullet_{rig}(U_k)$. Although in our special case the same vector spaces 
are also isomorphic to the cohomology of uncompleted de Rham complex 
$(\Omega^\bullet(A), d)$, the operator $F^*$ used in the computation of zeta 
function is naturally defined on the completed de Rham complex, and we need to use
the quasi-isomorphism of Theorem \ref{thm:isom}  to transport it to the usual de Rham complex
involved in the results of Dimca and Saito.

 Let us first consider the case when $n$ is odd (so $X$ has even dimension $(n-1)$). 
 By Theorem \ref{e2-theorem} the $E_2$ page of the pole spectral sequence has at 
 most $n$ nonzero terms $E_2^{n-s, s}$ for $s \in \{1, \ldots n\}$. 
   In order to find the zeta function of $f$ we must compute the Frobenius action of each basis element in these $n$ spaces. 
   Such basis element,  coming from  $
   h dx_0 \wedge \ldots \wedge dx_nH^{n+1}(K_f^{\bullet})_{sN-n-1}$ for $s\in\{1,
  \ldots, n\}$ can be written as
 \begin{equation*}
 	\frac{\Delta(hdx_0\wedge\dots\wedge dx_n)}{f^s}=\frac{h\Omega}{f^s}
 \end{equation*}
 where $\Omega = \Delta(dx_0\wedge\dots\wedge dx_n)$. 
 After applying the Frobenius action to the basis elements of each rigid cohomology module 
 we then reduce in cohomology and form the (square) matrix of Frobenius.  The coefficients of the characteristic polynomial of this matrix are integers and hence there $p$-adic representations are finite thereby allowing us to truncate the results and recover the zeta function.
 
 For even $n$ there is a single non-zero term on the diagonal $E_2^{t, s}$ with 
 $t+s = (n-1)$, for $s = n/2$ and the cohomology classes can be represented by closed 
 forms which are linear combinations of $\Delta(dx_0 \wedge \ldots \wedge
 \widehat{dx_i} \wedge \ldots \wedge dx_n)$ with coefficients $\frac{h_i}{f^{s^i}}$
 so the same procedure applies.

 \subsection{Formulas for the lift of Frobenius}
 
 Here we assume that $k = \mathbb{F}_p$ to simplify the notation
  (the case of field with $p^a$ elements follows by 
 a simple adjustment as in \cite{Ge})
 Following Section  4 of \cite{Ge}, 
 denote by $\sigma$ the Frobenius automorphism both in $Gal(\mathbb{F}_q/\mathbb{F}_p)$
 and in $Gal(K/\mathbb{Q}_p)$. 
  The Frobenius action is defined on coordinate 
  functions by  $\hat{F}_p(x_i)=x_i^p$.  To compute where $1/f$ is mapped to, use the equation
 \begin{equation} \label{eq:moveAdag}
 	1=\hat{F}_P(f/f)=f^{\sigma}(x^p)\hat{F}_p(1/f)
 \end{equation}
 where $f^{\sigma}$ is obtained by applying $\sigma$ to each coefficein of $f$.  
 Equation (\ref{eq:moveAdag}) has no solution in $A$ but in $A^{\dagger}$
 one can use 
 \begin{equation*}
 	f^{-p}(1-pg/f^p)^{-1},\text{ where }pg=f^p-f^{\sigma}(x^p).
 \end{equation*}
For a homogeneous polynomial $h$ of degree $sN - (n+1)$ this gives the formula
(4.1) in \cite{Ge}: 
 \begin{equation} \label{eq:FrobCh1}
 	\hat{F}_p: \frac{h\Omega}{f^l}\mapsto p^n\frac{h(x^p)(x_0x_1\ldots x_n)^{p-1}\Omega}{f^{pl}}\Bigg(\sum_{k=0}^{\infty}p^k\frac{\alpha_kg^k}{f^{pk}}\Bigg)
 \end{equation}
 where $\alpha_k$ is the $k^{\text{th}}$ coefficient of the power series $(1-t)^{-l}=1+\alpha_1t+\alpha_2t^2+\dots$.  This action extends to $A^\dagger$ by continuity.

  Note that the Frobenius action will increase the pole order 
   but we can reduce it in cohomology because the de Rham complex of the algebra $A$ has finite dimensional cohomology 
  and hence the same holds for the Monsky-Washnitzer complex of $A^{\dagger}$. 
  We are planning to truncate the infinite series to a finite sum, since the 
  eigenvalues of the Frobenius matrix have known absolute values which allows 
  to recover the contribution for large enough index $k$, without computing 
  the terms explicitly. 
   In other
  words, adjusting by a series of exact forms we can reduce the output of $\hat{F}_p$ to a
   finite expression. 
   
   However, as any  algorithm can only compute finitely many terms, we will need to truncate
   the above series.  We can do this because $\zeta(f,T)$ is a rational function and the reciprocal of the characteristic polynomial of the Frobenius matrix has integer coefficients, cf. 
   \cite{LS}.  Any positive integer will have a terminating $p$-adic expansion while any negative integer's $p$-adic expansion will trail off with $p-1$ since
 \begin{equation*}
 	-1=p-1+(p-1)p+(p-1)p^2+\dots.
 \end{equation*}
 Theorem 3.2 in \cite{Ge}  gives the following bound allowing to recover the zeta function modulo $p^D$ with
 \begin{equation} \label{eq:boundCoeff}
 	D\geq\lceil\log_p(2\gamma+1)\rceil\text{ where }\gamma:=\binom{b}{\lfloor b/2\rfloor}q^{b/2}
 \end{equation}
 provided that the sign of the determinant of the Frobenius matrix is known.  Here $b$ is the number of basis elements for the de Rham cohomology of $U$.  For our case, when $n$ is odd,
 $f$ has $\tau$ isolated ordinary double points, $b= b(n, N) - \tau$.  The above bound is based on the fact that the eigenvalues for the Frobenius matrix, in the smooth case, have absolute value $p^{(n-1)/2}$,   and the coefficients of the characteristic polynomial of the Frobenius matrix, which is a reciprocal polynomial, can be written in terms of symmetric polynomials in the eigenvalues/roots.
 If $e_i$ stands for the $i$-th elementary symmetric function,   
 then any monic polynomial $g$ of degree $k$ with roots $r_1,r_2,\dots,r_k$ can be written as
 \begin{equation*}
 	g=x^k-e_1(r_1\dots,r_k)x^{k-1}+e_2(r_1,\dots,r_k)x^{k-2}+\dots+(-1)^ke_k(r_1\dots,r_k).
 \end{equation*}
 Now when $n$ is odd and  $f$ has isolated ordinary double points, the complement $U_k=\mathbb{P}^3_k\setminus X_k$ along with the $E_2$ page having zero 
 $s+t = n-1$, satisfy the hypotheses of Theorem 5.2 on page 174 of \cite{ChSt}.
 Therefore the eigenvalues of the Frobenius matrix have absolute value $p^{(n-1)/2}$.  Hence the bound (\ref{eq:boundCoeff}) from \cite{Ge} holds.  The last bound needed is for truncating the series of the Frobenius action, (\ref{eq:FrobCh1}).  Corollary 4.2 on page 22 of \textit{loc. cit.} tells us that we can truncate the series $\sum_{k}p^k\alpha_kg^k/f^{pk}$, removing the terms with 
 $k \geq M$, as long as
 \begin{equation*}
 	k\geq D+ (n+1)\lfloor\log_p(p(k+n)-1)\rfloor-n+1+r_1\;\;\;\;\;\;\forall\;k\geq M
 \end{equation*}
 where $r_1=0$ for $p>2$, see the bottom of page 21 of \textit{loc. cit.}. Recall that $p > n-1$ for us.

 \subsection{Reduction in top cohomology.} 
 \label{reduction}
 The Frobenius action formula (\ref{eq:FrobCh1}) is an infinite series while a computer can only handle finitely many  terms.  For our algorithm we truncate this series and reduce each term until it is written in the cohomology classes of the basis elements for the $E_2$ page.  We now explain the reduction process.  Given $g\Omega/f^s$ for some $g\Omega\in S_{sN-n}\Omega^n$ we can write
 \begin{equation*}
 	g\Omega=\Delta(gdx_0\wedge dx_1\wedge \ldots \wedge dx_n)
 \end{equation*}
 where $gdx_0\wedge dx_1\wedge \ldots\wedge dx_n\in S_{sN-n-1}\Omega^{n+1}$.  If $s>n$ then the de Rham differential induces an isomorphism on the $E_1$ page by Theorem 
 \ref{e2-theorem} and $gdx_0\wedge\ldots\wedge dx_n$ can be expressed as
 \begin{equation} \label{eq:gao}
 	gdx_0\wedge dx_1\wedge \ldots\wedge dx_n=d\alpha+df\wedge\omega
 \end{equation}
 where
 \begin{gather}
 	 \label{eq:aInKer}
 	\alpha\in\ker(S_{sN-n}\Omega^n\xrightarrow{df\wedge}S_{(s+1)N-n-1}\Omega^{n+1});
 	\qquad 
 	\omega\in S_{(s-1)N-3}\Omega^3.
 \end{gather}
 Now consider
 \begin{align}
 	& \frac{g\Omega}{f^s}-d\Bigg(\frac{N}{|\omega|}\frac{\Delta(\omega)}{f^{s-1}}-\frac{\Delta(\alpha)}{f^s}\Bigg) \nonumber  =\frac{g\Omega}{f^s}+\frac{N}{|\omega|}\frac{\Delta(d_f(\omega))}{f^s}-\frac{\Delta(d_f(\alpha))}{f^{s+1}}\text{, by (\ref{eq:domegaFull})} \nonumber \\
 	& =\frac{g\Omega}{f^s}+\frac{N}{|\omega|}\frac{\Delta(d\omega)}{f^{s-1}}-\frac{\Delta(gdx_0\wedge\dots\wedge dx_n-d\alpha)}{f^{s}}-\frac{\Delta(d\alpha)}{f^s}\text{, by (\ref{eq:gao}) and (\ref{eq:aInKer})} \nonumber \\
 	& =\frac{g\Omega}{f^s}+\frac{N}{|\omega|}\frac{\Delta(d\omega)}{f^{s-1}}-\frac{g\Omega}{f^{s}}+\frac{\Delta(d\alpha)}{f^s}-\frac{\Delta(d\alpha)}{f^s} \nonumber =\frac{N}{|\omega|}\frac{\Delta(d\omega)}{f^{s-1}} \nonumber \\
 	& \Rightarrow \frac{g\Omega}{f^s}\equiv\frac{N}{|\omega|}\frac{\Delta(d\omega)}{f^{s-1}}\text{ (mod }d_{\text{dR}}). \label{eq:reducFormula}
 \end{align}
 
 Hence when reducing in cohomology we first check if the $g$ from (\ref{eq:gao}) is in the image of Koszul differential, $df\wedge$.  This can be determined by using a Groebner basis for the Jacobian ideal of $f$ (and in higher degrees by applying a lemma that follows).  Now if $g$ is in the image of $df\wedge$ then $\alpha=0$ and we can decrease the pole order of $g\Omega/f^s$ by using the reduction formula, (\ref{eq:reducFormula}).  If $g$ is not in the image of $df\wedge$ and $s>n$ then we must find the $\alpha$ from (\ref{eq:gao}) in order to reduce in cohomology.  Technically this problem can be solved with Linear Algebra by using matrices, that is one can construct matrices for the kernel and image of $df\wedge$, however this is not practical if the polynomial part of $\alpha$ is large. For example suppose that $n=3$, $\deg(f)=\deg(h)=4$ and $p=7$, then the third term ($k=2$) in the Frobenius action formula has polynomial degree equal to $4p+4(p-1)+4pk=108$.  There are $\binom{108+3}{3}=221,815$ monomials of degree 108 in four variables and this is only the third term in the series!  More importantly, even if one did construct a matrix of this size it would only reduce the pole order of $g\Omega/f^s$ by 1.  Thankfully there is a quicker way to find $\alpha$ which uses bases of the subdiagonal on the $E_1$ page. We will apply the following 
 
 \begin{lemma}
 	\label{top-koszul}
 	Let $J=(f_0,f_1,\ldots,f_n)$ be the Jacobian ideal of $f$ with $\deg(f)=N\geq2$ and suppose that the projective hypersurface defined by $f$ has isolated ordinary double points. 
 	  If $g\in\mathbb{C}[x_0,x_1,\ldots,x_n]$ with $\deg(g)\geq (n+1)(N-1)$, then $g\in J$ if and only if $g$ vanishes at all singular points of $f$.
 \end{lemma}

 \begin{proof}
 	Let $I=\sqrt{J}$ be the radical of the Jacobian ideal, $S=\mathbb{C}[x_0,x_1,\ldots,x_n]$, 
 	and assume that $Z(f)$ has transversal intersection with $Z(x_0)$.
 	Note that we can always find a change of coordinates such that $Z(x_0)$ is transversal to 
 	$X=Z(f)$,  for instance see Example 3.3 in \cite{DiSt}. 
 	
 	Next, consider the map $\Omega^n_{m+n} \to S_m/I_m$ defined by sending an $n$ form
 	$\alpha$ to its coefficient of $dx_1 \ldots dx_n$ reduced modulo $I$. By transversatlity 
 	assumption and Theorem 1.2 in \cite{DiSt}, this induces an embedding of 
 	$H^n(K_f^{\bullet})_m$ into $S_m/I_m$. 
   If $m\geq (n+1)(N-1)$ then by Theorem \ref{e2-theorem} 
 	\begin{align*}
 		\tau & =\dim(H^n(K_f^{\bullet})_m)  \leq\text{codim}(I_m), \qquad\text{ (from the embedding)} \\
 		& \leq\text{codim}(J_m), \hspace{3.4cm} (\text{since }J\subseteq I )\\
 		& =\tau, \hspace{4.8cm} (\text{since } \dim(H^{n+1}(K_f^{\bullet})_m=\tau).
 	\end{align*}
 	The above shows that, in this range, $I_m=J_m$.  Moreover, 
 	for a nodal hypersurface $g\in I$ if and only if $g$ vanishes at all singular points of $f$: this is a reformulation of Theorem 1.5 in  \cite{DiSt2} (its proof in \textit{loc. cit.} is based on Cayley-Bacharach Theorem).  This proves the lemma. 
 \end{proof}
It follows that in (\ref{eq:gao}), $g$ has the same values at the singular points (or rather their 
fixed lifts $P_1, \ldots, P_\tau$ to vectors in the $n+1$-dimensional affine space) as the
polynomial $d\alpha/(dx_0 \wedge \ldots \wedge dx_n)$, where $\alpha \in H^n(K_f^{\bullet})_{sN-n}$
By Theorem \ref{e2-theorem}, for $s > n$ we have 
$\dim(H^n(K_f^{\bullet})_{(n+1)N-n})=\tau$. Let $\{\gamma_1,\dots,\gamma_{\tau}\}$ be a basis for this space.  Further, suppose that the coordinate hyperplane $x_0=0$ is transversal to $X$ then
\begin{equation*}
	\{x_0^k\gamma_1,\dots,x_0^k\gamma_{\tau}\}
\end{equation*}
is a basis for $H^n(K_f^{\bullet})_{nN-n+k}$ for any $k\geq0$ by Corollary 11 of \cite{ChDi}.   At this point it is clear that $\alpha$ will be a linear combination of $x_0^k\gamma_1\dots,x_0^k\gamma_{\tau}$ for $k=(s-n-1)N$ 
and the coefficients of that linear combination are computed by looking 
at $d(x_i^k \gamma_j)$ and computing their values at $P_1, \ldots, P_\tau$.

 More precisely, let us look at
  $H^3(K_f^{\bullet})_{(n+1)N-n+k}$ with $k \geq 0$.  
  We know that de Rham differential induces and isomorphism and therefore
 \begin{equation*}
 	\{d(x_0^k\gamma_1),\dots,d(x_0^k\gamma_{\tau})\}
 \end{equation*}
 is a basis for $H^4(K_f^{\bullet})_{(n+1)(N-1)+k}$.  In particular $d(x_0^k\gamma_i)$ is not in the image of $df\wedge\Rightarrow$ the polynomial part of $d(x_0^k\gamma_i)$ is not in the Jacobian ideal and from the lemma above $d(x_0^k\gamma_i)$ does not vanish at all singular points of $f$.  We then have the following linear system of equations
 \begin{align}
 	c_1d(x_0^k\gamma_1)|_{P_1}+\dots+c_{\tau}d(x_0^k\gamma_{\tau})|_{P_1} & =g|_{P_1} \nonumber \\
 	c_1d(x_0^k\gamma_1)|_{P_2}+\dots+c_{\tau}d(x_0^k\gamma_{\tau})|_{P_2} & =g|_{P_2} \nonumber \\
 	& \;\;\vdots \nonumber \\
 	c_1d(x_0^k\gamma_1)|_{P_{\tau}}+\dots+c_{\tau}d(x_0^k\gamma_{\tau})|_{P_{\tau}} & =g|_{P_{\tau}}. \label{eq:sysRed}
 \end{align}
 This system has a unique solution, for if it did not then there would be a non trivial solution to
 \begin{equation*}
 	c_1v_1+\dots+c_{\tau}v_{\tau}=0\text{ where }v_i=\begin{pmatrix}
 		d(x_0^k\gamma_i)|_{P_1} \\
 		\vdots \\
 		d(x_0^k\gamma_i)|_{P_{\tau}}
 	\end{pmatrix}.
 \end{equation*}
 This means that the polynomial part of $c_1d(x_0^k\gamma_1)|_{P_1}+\dots+c_{\tau}d(x_0^k\gamma_1)|_{P_{\tau}}$ vanishes at all singular points of $f$ and hence it is in the Jacobian ideal of $f$ from the lemma above.  But this is a contradiction since the $d(x_0^k\gamma_i)$ are a basis for $H^{n+1}(K_f^{\bullet})_{(n+1)(N-1)+k}$ and by definition not in the image of $df\wedge$.  Therefore the system of equations (\ref{eq:sysRed}) has a unique solution.  Furthermore we can write the solution as
 \begin{equation} \label{eq:solRed}
 	\begin{pmatrix}
 		c_1 \\
 		\vdots \\
 		c_{\tau}		
 	\end{pmatrix}=\begin{pmatrix}
 		d(x_0^k\gamma_1)|_{P_1} & \dots & d(x_0^k\gamma_{\tau})|_{P_1} \\
 		\vdots & \ddots & \vdots \\
 		d(x_0^k\gamma_1)|_{P_{\tau}} & \dots & d(x_0^k\gamma_{\tau})|_{P_{\tau}}
 	\end{pmatrix}^{-1}\begin{pmatrix}
 		g|_{P_1} \\
 		\vdots \\
 		g|_{P_{\tau}}
 	\end{pmatrix}
 \end{equation}
 where the $\tau\times\tau$ matrix can be reused anytime one is reducing in cohomology in this degree.  That is the square matrices used to reduce the pole order for one basis element, $h$,  do not have to be recalculated.  This completes our method of finding an $\alpha$ such that $gdx_0\wedge\dots\wedge dx_3=d\alpha+df\wedge\omega$ and this is one of the main differences between the ordinary double point case and the smooth case.
 
 \subsection{Reduction on the subdiagonal for even $n$.} 
 
 Since $E_1^{t, s}$ has no nonzero terms with $t+s = n-2$ and $E_2^{t, s}$ has 
 a unique nonzero term for $t+s = n-1$, the situations is much simpler: in 
 (\ref{eq:gao}) we can simply take $\alpha = 0$ and then use the analogue 
 of (\ref{eq:reducFormula}) to reduce the pole order while keeping the same
 cohomology class.

 \section{Steps of the algorithm}
 
 Below is a brief description of each step of our algorithm.  
 \begin{enumerate}
 	\item[1.] On the $E_1$ page compute $(H^n(K^{\bullet}_f)_{jN-n},H^{n+1}(K^{\bullet}_f)_{jN-n-1})$ for $j=1,\ldots, n$.
 	\item[2.] Find a basis for every cohomology space in each pair from step 1.
 	\item[3.] For odd $n$, use the bases from step 2 (each pair) to calculate a basis for the top cohomology spaces on the $E_2$ page of the spectral sequence
 	\begin{equation*}
 		\frac{b_1\Omega}{f^{i_{b_1}}},\dots,\frac{b_L\Omega}{f^{i_{b_L}}},\text{ where }L=\frac{1}{N}((N-1)^{n+1}+(N-1))-\tau.
 	\end{equation*}
    For even $n$, calculate the basis for $E_2^{n-s, s}$ with $s = 1, \ldots, n$ and 
    for $E_2^{n/2-1, n/2}$. 
 	\item[4.] Calculate Frobenius action for every basis element from step 3 using the formula (\ref{eq:FrobCh1}) and its analogue for the subdiagonal in which 
 	$p^{n+1} (x_0x_1\ldots x_n)^{p-1}\Omega$ is replaced by Frobenius images of 
 	$\Delta(dx_0 \wedge \ldots \wedge
 	\widehat{dx_i} \wedge \ldots \wedge dx_n)$ and the sum of 
 	(\ref{eq:FrobCh1}) is applied to each of the coefficients $\frac{h_i}{f^{s^i}}$.
 	\item[5.] Reduce in cohomology as in Section \ref{reduction}.
 	\item[6.] Use Step 5 to compute  the Frobenius matrix and calculate its characteristic polynomial.
 \end{enumerate}

 Any algorithm must imply a finite computation. Results on spectral sequences
 given in Section 4  tell us that the $E_2$ page vanishes when the  total degree is
  greater than $nN$.  Thus we can restrict our attention to the terms of the $E_1$ page
  mentioned in Step 1. 
 Finding their dimension and then a basis, if de Rham is not surjective, can be done with a computer by using linear algebra since the differentials $d$ and $df\wedge$ are linear operators.  This also limits the number of spaces on the $E_2$ page for step 3.  Again we are only interested in $E_2^{t, s}$ for $s=1,\ldots,n$ and $t = n-s$, and for even $n$ also $s = n/2$, $t = n/2-1$.  Step 4 is a formula which one can easily implement with a computer and bounds for when to truncate were given in Section 5.2.  In the next section we explain step 5.

 At the end of step 5 we have written the (truncated) Frobenius action of each basis element as a linear combination of all basis elements ($E_2$ page).  If $f\in\mathbb{Z}[x_0,x_1,x_2,x_3]$ then this linear combination is over $\mathbb{Q}$.  For practical purposes we can construct the Frobenius matrix such that its entries are these constants in $\mathbb{Q}$ or their $p$-adic expansion.  In any case we must calculate the $p$-adic expansion of the coefficients of the reciprocal characteristic polynomial of the Frobenius matrix.  Explicit examples are given in 
 the next section and for more on $p$-adic numbers see \cite{Ko}.

 \section{Examples of computations for surfaces ($n=3$). }
 
 We apply our algorithm to the Cayley Cubic, a Kummer Surface, a quartic with 6 ordinary double points and a quintic with 14 ordinary double points.  For the Cayley Cubic we go in to greater detail.

\subsection{Cayley Cubic}
The defining equation for the Cayley Cubic is 
\begin{equation}
	f(x_0, x_1, x_2, x_3) = x_0x_1x_2+x_0x_1x_3+x_0x_2x_3+x_1x_2x_3=0. \label{eq:CayleyDef}
\end{equation}
This surface has four ordinary double points, the maximum for any surface of degree three, located at $[1:0:0:0]$, $[0:1:0:0]$, $[0:0:1:0]$, and $[0:0:0:1]$.  It is important to note that using our algorithm for the Cayley Cubic is overkill.      We have the following table.

\begin{table}[h]
	\centering
	\renewcommand{\arraystretch}{1.2}
	\begin{tabular}{|c|c|}
		\hline
		point(s) & number of roots of $f$ \\
		\hline
		\hline
		$[0:0:0:1]$ & $1$ \\
		\hline
		$[0:0:1:x_3]$ & $q$ \\
		\hline
		$[0:1:x_2:x_3]$ &  $2q-1$ \\
		\hline
		$[1:x_1:x_2:x_3]$ & $q^2+1$ \\
		\hline
	\end{tabular}
	\caption{Rational Points of the Cayley Cubic}
	\label{CayleyRoots}
\end{table}

For the last row of Table \ref{CayleyRoots} we go into more detail.  On $[1:x_1:x_2:x_3]$ the surface becomes
\begin{equation}
	x_1(x_2+x_3+x_2x_3)+x_2x_3=0 \label{eq:A3}
\end{equation}
\textbf{Case 1}: $x_2+x_3+x_2x_3=0$ then (\ref{eq:A3}) reduces to $x_2x_3=0$.  
Subtracting  gives $x_2+x_3=0$ and $x_2x_3=0\Rightarrow x_2=x_3=0$ giving us $q$ roots since $f(1,x_1,0,0)=0$.

\noindent
\textbf{Case 2}: $x_2+x_3+x_2x_3\neq0$ then we can explicitly solve for $x_1$ in (\ref{eq:A3}) yielding
\begin{equation*}
	x_1=-\frac{x_2x_3}{x_2+x_3+x_2x_3}.
\end{equation*}
  How many points $(x_2,x_3)$ are there such that $x_2+x_3+x_2x_3\neq0$? This reduces
   to the question: how  many points yield $x_2+x_3+x_2x_3=0$? For $x_3 \neq -1$ 
   we can rewrite $x_2=-\frac{x_3}{1+x_3}$. This gives  $q-1$ points $(x_2,x_3)$ such that $x_2+x_3+x_2x_3=0$.  The value $x_3 = -1$ is not compatible with $x_2 + x_3 + x_2 x_3 = 0$. 
   
   This means that i case 2 there are $q^2-(q-1)$ roots of $f$.  Adding the number of roots in cases 1 and 2 gives $q+q^2-(q-1)=q^2+1$.  Since $q$ was a arbitrary power of some prime $p$, we have $\#V(\mathbb{F}_{q^r})=1+3q^r+q^{2r}$ (this is the sum of the second column of Table \ref{CayleyRoots} with $q$ replaced by $q^r$). Thus
\begin{equation*}
	\zeta(f,T)=\exp\Bigg(\sum_{r=1}^{\infty}(q^{2r}+3q^r+1)\frac{T^r}{r}\Bigg)=\frac{1}{(1-T)(1-qT)^3(1-q^2T)}.
\end{equation*}
To compare this with the output of our algorithm, first note that $f$ is equisingular since  the coordinates of the singular points consist of only ones or zeros which are in any  field.  Now we look at the $E_0$ page of the spectral sequence. 
We use Figure \ref{fig:E0page} with $N=n=3$. Starting with the top diagonal we have
\begin{align}
	\dim(H^4(K_f^{\bullet})_2) 
	& =10-\dim(\text{im}(S_0\Omega^3\xrightarrow{df\wedge}S_2\Omega^4)) \label{eq:dimIm}
\end{align}
By straightforward linear algebra we get that   
\begin{equation} \label{eq:E1top}
	\{x_0^2,x_0x_1,x_0x_2,x_1^2,x_2^2,x_3^2\}dx_0\wedge dx_1\wedge dx_2\wedge dx_3
\end{equation}
project to a basis for $H^4(K^{\bullet}_f)_2$.  For the subdiagonal space $H^3(K^{\bullet}_f)_3$ we have
\begin{align*}
	\dim(H^3(K^{\bullet}_f)_3) 
	& =\dim(\ker(S_3\Omega^3\xrightarrow{df\wedge}S_5\Omega^4))-\dim(\text{im}(S_1\Omega^2\xrightarrow{df\wedge}S_3\Omega^3)).
\end{align*}
Using Mathematica we find that $\dim(H^3(K^{\bullet}_f)_3)=4$ and in general
\begin{table}[h]
	\centering
	\renewcommand{\arraystretch}{1.2}
	\begin{tabular}{|c|c|c|c|c|c|c|c|c|}
		\hline
		$j$ & 0 & 1 & 2 & 3 & 4 & 5 & 6 & $\dots$ \\
		\hline
		$\dim(H^4(K^{\bullet}_f)_j)$ & 1 & 4 & 6 & 4 & 4 & 4 & 4 & $\dots$ \\
		\hline
		$\dim(H^3(K^{\bullet}_f)_j)$ & 0 & 0 & 3 & 4 & 4 & 4 & 4 & $\dots$ \\
		\hline
	\end{tabular}
	\caption{Koszul Cohomology Dimensions of the Cayley Cubic}
	\label{KoszulDim}
\end{table}
The first non trivial kernel of $df\wedge:S_j\Omega^3\rightarrow S_{j+2}\Omega^4$ occurs when $j=2$ and it can be used to give a basis for higher cohomology groups by multiplying the basis elements by certain powers of $x_i$.  For example Mathematica tells us that the following 3-forms are a basis for $H^3(K^{\bullet}_f)_2$
\begin{align*}
	n_1 & =(2x_0x_2+3x_2x_3-2x_3^2)dx_0\wedge dx_1\wedge dx_2-(2x_0x_3-2x_2^2+3x_2x_3)dx_0\wedge dx_1\wedge dx_3 \\
	& \;\;\;\;-(6x_0x_1+2x_0x_2+2x_0x_3-x_2x_3)dx_0\wedge dx_2\wedge dx_3-(4x_0^2-x_2x_3)dx_1\wedge dx_2\wedge dx_3 \\
	n_2 & =(2x_0x_1+3x_1x_3-2x_3^2)dx_0\wedge dx_1\wedge dx_2 \\
	& \;\;\;\;+(2x_0x_1+3x_0x_2-x_0x_3-x_1x_3-3x_2x_3)dx_0\wedge dx_1\wedge dx_3 \\
	& \;\;\;\;-(3x_0x_1+x_0x_3+2x_1^2)dx_0\wedge dx_2\wedge dx_3-(4x_0^2-x_1x_3)dx_1\wedge dx_2\wedge dx_3 \\
	n_3 & =(x_0x_1+x_0x_2+x_1x_2+3x_1x_3+3x_2x_3)dx_0\wedge dx_1\wedge dx_2 \\
	& \;\;\;\;+(x_0x_1+3x_0x_2+2x_2^2)dx_0\wedge dx_1\wedge dx_3-(3x_0x_1+x_0x_2+2x_1^2)dx_0\wedge dx_2\wedge dx_3 \\
	& \;\;\;\;-(4x_0^2-x_1x_2)dx_1\wedge dx_2\wedge dx_3.
\end{align*}
For example, one can check that $\{x_0n_1,x_2n_1,x_3n_1,x_1n_2\}$ is a basis for $H^3(K^{\bullet}_f)_3$.  We are taking advantage of the fact that $df\wedge$ is $\mathbb{C}[x_0,x_1,x_2,x_3]$-linear for if $df\wedge n_1=0\Rightarrow df\wedge(x_0n_1)=x_0(df\wedge n_1)=0$ and therefore we only need to check if this element is in the image of $df\wedge$, where Lemma \ref{top-koszul} applies.  

Steps 1 and 2 of the algorithm have been completed.  For Step 3 we compute a basis for the only cohomology space (for this example) on the $E_2$ page
\begin{align*}
	E_2^{4,2} 
	& =H^4(K^{\bullet}_f)_2/\text{im}(H^3(K^{\bullet}_f)_3\xrightarrow{d}H^4(K^{\bullet}_f)_2).
\end{align*}
In order to find a basis for this space we must calculate the de Rham differential of each element in the basis $\{x_0n_1,x_2n_1,x_3n_1,x_1n_2\}$.  Take the element $x_0n_1$
\begin{equation*}
	d(x_0n_1)=(-6x_0^2+x_0x_2+x_0x_3+x_2x_3)dx_0\wedge dx_1\wedge dx_2\wedge dx_3.
\end{equation*}
Recall that (\ref{eq:E1top}) is our basis for $H^4(K^{\bullet}_f)_2$ and not all of the monomials of $d(x_0n_1)$ are in this basis, only $x_0^2$.  To solve this problem we can add an element in the $\text{im}(P_0\Omega^3\xrightarrow{df\wedge}P_2\Omega^4)$ to $x_0n_1$. For example

\begin{align*}
	d\Big( (-x_0n_1+df\wedge(x_3dx_0\wedge dx_2))/6\Big) & =x_0^2dx_0\wedge dx_1\wedge dx_2\wedge dx_3.\\
	d\Big((x_2n_1+2df\wedge(x_3dx_0\wedge dx_2))/3\Big) & =x_2^2dx_0\wedge dx_1\wedge dx_2\wedge dx_3 \\
	d\Big((x_3n_1+2df\wedge(x_3dx_0\wedge dx_2))/3\Big) & =x_3^2dx_0\wedge dx_1\wedge dx_2\wedge dx_3 \\
	d\Big((x_1n_2+df\wedge(x_3dx_0\wedge dx_1))/3\Big) & =x_1^2dx_0\wedge dx_1\wedge dx_2\wedge dx_3
\end{align*}
It's obvious that $E_2^{4,2}$ is  spanned by (images of) $\{x_0x_1,x_0x_2\}dx_0\wedge\dots\wedge dx_3$.  
Now we calculate the action of Frobenius on
	$\frac{x_0x_1\Omega}{f^2},\frac{x_0x_2\Omega}{f^2}$.
For $h=x_0x_1$ the equation (\ref{eq:FrobCh1}) becomes
\begin{equation} \label{eq:CayleyFrob}
	p^3\frac{(x_0x_1)^p(x_0x_1x_2x_3)^{p-1}\Omega}{f^{2p}}\Bigg(\sum_{k=0}^{\infty}\frac{(k+1)(f^p-f(x^p))^k}{f^{pk}}\Bigg).
\end{equation}
As an example, let $p=5$, then the first term ($k=0$) in the series above is $125x_0^9x_1^9x_2^4x_3^4\Omega/f^{10}$.  Our final goal is to write this form as a linear combination of $x_0x_1\Omega/f^2$ and $x_0x_2\Omega/f^2$, but first we must reduce it to $\beta/f^9$ where $\beta\in P_{24}\Omega^3$.  To do this we write $125x_0^9x_1^9x_2^4x_3^4dx_0\wedge\dots\wedge dx_3$ in the form of (\ref{eq:gao}) which can be accomplished by using a Groebner basis for the Jacobian of $f$.  Indeed we find that
$125x_0^9x_1^9x_2^4x_3^4dx_0\wedge\dots\wedge dx_3=df\wedge\omega$
where
\begin{align}
	\omega & =125x_0^9x_1^9x_2^3x_3^2\Bigg[\frac{1}{6}x_3dx_0\wedge dx_1 \wedge dx_2-\Bigg(\frac{2}{3}x_2+\frac{1}{2}x_3\Bigg)dx_0\wedge dx_1\wedge dx_3 \nonumber \\
	& \;\;\;\; -\Bigg(\frac{1}{3}x_1+\frac{1}{2}x_3\Bigg)dx_0\wedge dx_2\wedge dx_3+\Bigg(\frac{1}{3}x_0+\frac{1}{2}x_3\Bigg)x_0dx_1\wedge dx_2\wedge dx_3\Bigg]. \label{eq:omegaCayley}
\end{align}
Two remarks are important here.  One, the 3-form $\omega$ is not unique, and two, we did not have to use the de Rham differential since the monomial $x_0^9x_1^9x_2^4x_3^4$ is in the Jacobian ideal of $f$, that is, the $\alpha$ of (\ref{eq:gao}) is zero in this case.  Back to the reduction we have
\begin{equation*}
	\frac{125x_0^9x_1^9x_2^4x_3^4\Omega}{f^{10}}\equiv\frac{1}{9}\frac{\Delta(d\omega)}{f^9}(\text{mod }d_{\text{dR}})
\end{equation*}
the $\omega$ from above, (\ref{eq:omegaCayley}).  For the next reduction we must find an $\omega_1\in S_{21}\Omega^3$ such that $df\wedge\omega_1=d\omega$. 
Again we check if $d\omega\in\text{im}(S_{21}\Omega^3\xrightarrow{df\wedge}S_{23}\Omega^4)$ and it is.  When the 4-form is not in the image of the Koszul differential, for the case of the  Cayley Cubic, the remainder can be discarded because it is a pure power of $x_i$ which is in the image of $d$.  By Lemma \ref{top-koszul} any mixed monomial of degree 3 is in the Jacobian ideal of $f$. Hence when we are reducing in cohomology, if there is a remainder, then $\alpha$ from (\ref{eq:gao}) can only be $x_i^{3k+2}dx_0\wedge\dots\wedge dx_3$ for $i=0,1,2,3$ and $k\in\mathbb{Z}_{>0}$.  We have already given explicit examples of 3-forms whose de Rham differential is $x_i^2dx_0\wedge\dots\wedge dx_3$ for $i=0,1,2,3$ and the general case admits
similar formulas.
The  table below shows the reductions of the first 7 terms of the series in (\ref{eq:CayleyFrob}).
\begin{table}[h]
	\begin{center}
		\renewcommand{\arraystretch}{1.6}
		\begin{tabular}{|c|c|}
			\hline			
			$k$ & $k^{\text{th}}$ part of the reduction of $x_0x_1$ \\
			\hline
			0 & $\frac{1}{126}\cdot5^2x_0x_1$ \\
			\hline
			1 & $\frac{1}{9009}\cdot5^5x_0x_1$ \\
			\hline
			2 & $\frac{1}{34034}\cdot5^6x_0x_1$ \\
			\hline
			3 & $\frac{1013}{5819814}\cdot5^5x_0x_1$ \\
			\hline
			4 & $\frac{1487}{38244492}\cdot5^6x_0x_1$ \\
			\hline
			5 & $\frac{2084}{49766871}\cdot5^6x_0x_1$ \\
			\hline
			6 & $\frac{2087}{1185579252}\cdot5^8x_0x_1$ \\
			\hline
		\end{tabular}
	\end{center}
	\label{redResults}
\end{table}

For the other basis element $x_0x_2dx_0\wedge\cdots\wedge dx_3$ the reductions are exactly the same; meaning that the table for $x_0x_2$ is the same as the one above, but with $x_2$ instead of $x_1$.  Let
\begin{equation*}
	r_0=\frac{5^2}{126}, r_1=\frac{5^5}{9009}, r_2=\frac{5^6}{34034},\dots, r_6=\frac{2087\cdot5^8}{1185579252}
\end{equation*}
and let $\phi$ be the embedding of $\mathbb{Q}$ into $\mathbb{Q}_5$.  We now look at the convergence of $\phi(r_0+\dotsc+r_i)$.
\begin{align*}
	\phi(r_0) & =1\cdot5^2+0\cdot5^3+0\cdot5^4+4\cdot5^5+4\cdot5^6+4\cdot5^7+0\cdot5^8+0\cdot5^9+\dots \\
	\phi(r_0+r_1) & =1\cdot5^2+0\cdot5^3+0\cdot5^4+3\cdot5^5+2\cdot5^6+0\cdot5^7+0\cdot5^8+0\cdot5^9+\dots \\
	\phi(r_0+r_1+r_2) & =1\cdot5^2+0\cdot5^3+0\cdot5^4+3\cdot5^5+1\cdot5^6+3\cdot5^7+4\cdot5^8+4\cdot5^9+\dots \\
	\phi(r_0+\dotsc+r_3) & =1\cdot5^2+0\cdot5^3+0\cdot5^4+0\cdot5^5+0\cdot5^6+2\cdot5^7+0\cdot5^8+0\cdot5^9+\dots \\
	\phi(r_0+\dotsc+r_4) & =1\cdot5^2+0\cdot5^3+0\cdot5^4+0\cdot5^5+1\cdot5^6+4\cdot5^7+4\cdot5^8+4\cdot5^9+\dots \\
	\phi(r_0+\dotsc+r_5) & =1\cdot5^2+0\cdot5^3+0\cdot5^4+0\cdot5^5+0\cdot5^6+0\cdot5^7+4\cdot5^8+0\cdot5^9+\dots \\
	\phi(r_0+\dotsc+r_6) & =1\cdot5^2+0\cdot5^3+0\cdot5^4+0\cdot5^5+0\cdot5^6+0\cdot5^7+0\cdot5^8+2\cdot5^9+\dots
\end{align*}
This sequence of numbers is 5-adically converging to 25 (we are using the 
estimate (\ref{eq:boundCoeff}) here).  Therefore the matrix of Frobenius is $25I_2$, and the ``interesting" part of the zeta function is
\begin{equation*}
	\det\Bigg(\begin{pmatrix}
		1 & 0 \\
		0 & 1
	\end{pmatrix}-5^{-1}T\begin{pmatrix}
		25 & 0 \\
		0 & 25
	\end{pmatrix}\Bigg)=(1-5T)^2.
\end{equation*}
This gives the zeta function of the Cayley Cubic for $p=5$ is
\begin{equation*}
	\zeta(f,T)=\frac{1}{(1-T)(1-5T)^3(1-25T)}.
\end{equation*}

\subsection{Kummer Surfaces}
A Kummer Surface is a degree 4 surface in $\mathbb{P}^3$ with 16 ordinary double points, the maximum number of nodes for any quartic.  Its family of defining polynomials are
\begin{equation*}
	(x_0^2+x_1^2+x_2^2-\mu ^2x_3^2)^2-\lambda(x_3-x_2-\sqrt{2}x_0)
	(x_3-x_2+\sqrt{2}x_0)(x_3+x_2+\sqrt{2}x_1)(x_3+x_2-\sqrt{2}x_1)
\end{equation*}
where $\lambda=(3\mu ^2-1)/(3-\mu ^2)$ and $\mu ^2 \neq \frac{1}{3}$, $\mu ^2 \neq 1$, $\mu ^2 \neq 3$.  In this section we find the zeta function of a Kummer Surface with $\mu=2$, that is
\begin{equation*}
	f=x_0^4+x_1^4+12x_2^4+27x_3^4+x_0^2(46x_1^2-20x_2^2-44x_2x_3-30x_3^2)-x_1^2(20x_2^2-44x_2 x_3+30x_3^2)-30x_2^2x_3^2.
\end{equation*}
The 16 ordinary double points of $f$ are
\begin{gather}
	\Big[\pm(\sqrt{2}+\sqrt{3}):\pm(\sqrt{2}-\sqrt{3}):-\sqrt{6}:2\Big],\Big[\pm(\sqrt{2}-\sqrt{3}):\pm(\sqrt{2}+\sqrt{3}):\sqrt{6}:2\Big], \nonumber \\
	\big[\pm\sqrt{3}:0:-1:1\big],\big[\pm3\sqrt{2}:0:-3:2\big],\big[0:\pm3\sqrt{2}:3:2\big],\big[0:\pm\sqrt{3}:1:1\big]. \label{eq:KummerSingPts}
\end{gather}
Choose a prime $p$ and see if the lift from $\mathbb{F}_p$ from $\mathbb{Z}_p$ is equisingular.  Using a Groebner basis one can show that $f$ has 16 singularities over the algebraic closures of $\mathbb{F}_5$ and $\mathbb{F}_7$, but things go wild for $\mathbb{F}_{11}$.  Take for instance the singular point $[0:\sqrt{3}:1:1]$ and define $g(x_0,x_1,x_2,x_3)=f(x_3,x_0\sqrt{3}+x_2,x_0+x_1,x_0)$.  Then $[1:0:0:0]$ is a singular point of $g$ and with the  quadratic part of $g(1,x_1,x_2,x_3)$  written as
\begin{equation*}
	-18x_1^2+8\sqrt{3}x_1x_2+12x_2^2+44x_3^2=(x_1\;\;x_2\;\;x_3)\begin{pmatrix}
		-18 & 4\sqrt{3} & 0 \\
		4\sqrt{3} & 12 & 0 \\
		0 & 0 & 44
	\end{pmatrix}\begin{pmatrix}
		x_1 \\
		x_2 \\
		x_3
	\end{pmatrix}
\end{equation*}
where the determinant of the $3\times3$ matrix above is $-11616=-2^5\cdot3\cdot11^2$.  This was just one of the 16 ordinary double points of $f$, but one can check that for all 16 points the corresponding determinant of the matrix above will be divisible by $11^2$ and therefore we exclude the case $p=11$ when using our algorithm.  However we can still calculate the zeta function of $f$ over $\mathbb{F}_{11}$ since
\begin{gather*}
	f\equiv 
	(x_0^2+x_1^2+x_2^2+7x_3^2)^2\in\mathbb{F}_{11}[x_0,x_1,x_2,x_3].
\end{gather*}
Thus, the zeta function of $f$ equals that of the smooth quadratic $x_0^2+x_1^2+x_2^2+7x_3^2$ which is
\begin{equation*}
	\zeta(f,T)=\frac{1}{(1-T)(1-11T)(1-121T)(1+11T)}.
\end{equation*}
Now apply our algorithm to $f$, $p=5, 7$.   First, $\dim(E_2^{4,0})=1$ and $dx_0\wedge\dots\wedge dx_3$ gives a basis for this space.  Moving up one level, to total degree $2N=8$, we use the following basis for $H^4(K_f^{\bullet})_4$
\begin{gather}
	\{x_0^4,x_0^3x_1,x_0^3x_2,x_0^3x_3,x_0^2x_1^2,x_0^2x_1x_2,x_0^2x_1x_3,x_0^2x_2^2,x_0x_1^3,x_0x_1^2x_2,x_0x_1^2x_3, \nonumber \\
	x_0x_1x_2^2,x_1^4,x_1^3x_2,x_1^3x_3,x_2^4,x_2^3x_3,x_2x_3^3,x_3^4\}dx_0\wedge dx_1\wedge dx_2\wedge dx_3. \label{eq:KummerTop}
\end{gather}
For the subdiagonal, $H^3(K_f^{\bullet})_5$, the authors were able to find a basis $\{b_1,\dots,b_{15}\}$ such that
\begin{align}
	d(b_1) & =x_0^3x_1  & d(b_5) & =x_0^2x_1x_3 & d(b_9) & =x_1^3x_2 & d(b_{13}) & =x_0^4-x_1^4 \nonumber \\
	d(b_2) & =x_0^3x_2 & d(b_6) & =x_0x_1^3 & d(b_{10}) & =x_1^3x_3 & d(b_{14}) & =3x_0^2x_1^2+x_1^4-2x_2^4 \nonumber \\
	d(b_3) & =x_0^3x_3 & d(b_7) & =x_0x_1^2x_2 & d(b_{11}) & =x_2^3x_3 & d(b_{15}) & =4x_2^4-9x_3^4 \nonumber \\
	d(b_4) & =x_0^2x_1x_2 & d(b_8) & =x_0x_1^2x_3 & d(b_{12}) & =x_2x_3^3 \label{eq:KummerSub}
\end{align}
( all polynomials above are multiplied by the factor $dx_0\wedge dx_1\wedge dx_2\wedge dx_3$ which we omit to save room).  Therefore
\begin{equation*}
	\{x_0^4,x_0^2x_1^2,x_0^2x_2^2,x_0x_1x_2^2\}dx_0\wedge dx_1\wedge dx_2\wedge dx_3
\end{equation*}
project to a basis for $E_2^{4,4}$.  The degree of the interesting part of the zeta function, $P(T)$, is
\newline $1/4((4-1)^{3+1}+4-1)-16=5$ and we have our 5 basis elements the calculate to Frobenius action on
\begin{equation*}
	\frac{x_0^4\Omega}{f^2},\frac{x_0^2x_1^2\Omega}{f^2},\frac{x_0^2x_2^2\Omega}{f^2},\frac{x_0x_1x_2^2\Omega}{f^2},\frac{\Omega}{f}.
\end{equation*}
For reducing in cohomology  we use the following basis for $H^3(K_f^{\bullet})_{jN-3}$
\begin{gather*}
	\{x_0^{4j-6}n_1,x_0^{4j-7}x_1n_1,x_0^{4j-7}x_2n_1,x_0^{4j-7}x_3n_1,x_0^{4j-8}x_1^2n_1,x_0^{4j-8}x_1x_2n_1,x_0^{4j-8}x_1x_3n_1,x_0^{4j-8}x_2^2n_1, \\
	x_0^{4j-8}x_2x_3n_1,x_0^{4j-9}x_1^3n_1,x_0^{4j-9}x_1^2x_2n_1,x_0^{4j-9}x_1^2x_3n_1,x_1^{4j-6}n_2,x_1^{4j-7}x_2n_2,x_1^{4j-7}x_3n_2,x_1^{4j-8}x_2^2n_2\}
\end{gather*}
where
\begin{align*}
	n_1 & =-x_0(2x_0^2+6x_1^2-20x_2^2-24x_2x_3-5x_3^2)dx_0\wedge dx_1\wedge dx_2 \\
	& \;\;\;\;-5x_0x_3(5x_2+6x_3)dx_0\wedge dx_1\wedge dx_3+5x_0x_1(4x_2+5x_3)dx_0\wedge dx_2\wedge dx_3 \\
	& \;\;\;\;+(x_0^2(22x_2+30x_3)-x_1^2(2x_2-5x_3)-9x_3^3)dx_1\wedge dx_2\wedge dx_3 \\
	n_2 & =(x_0^2(4x_2+5x_3)+x_1^2(4x_2-5x_3))dx_0\wedge dx_1\wedge dx_2 \\
	& \;\;\;\;+(x_0^2(5x_2+6x_3)-x_1^2(5x_2-6x_3))dx_0\wedge dx_1\wedge dx_3 \\
	& \;\;\;\;+x_1(6x_2^2-9x_3^2)dx_0\wedge dx_2\wedge dx_3+x_0(6x_2^2-9x_3^2)dx_1\wedge dx_2\wedge dx_3
\end{align*}
and $j=3,4,\dots,42$.  We know that this is a basis for these values of $j$ because we used a computer to check that the determinant of matrix in (\ref{eq:solRed}) is non zero.  The value $j=42$ comes from reducing the first five terms of the Frobenius action series with $p=7$, i.e., for $k=4,4p+4(p-1)+pkN=164=42N-4$.  For $p=5$ our algorithm gives $P(T)=(1-5T)(1+5T)^4$ which is easy to guess if one does a point count using Magma.  However things get more interesting for $p=7$ so we explain this case in greater detail.  As mentioned above, for $p=7$ we can recover $\zeta(f,T)$ with the reduction of the first five terms (technically 3 is enough) of the Frobenius action, however the fractions involved are too long to write below.  So we give the Frobenius matrix whose $p$-adic entries have been truncated at $7^8$
\begin{equation*}
	\begin{pmatrix}
		2932436 & 3752975 & 2573683 & 0 & 3187818 \\
		3326797 & 160280 & 4878860 & 0 & 5469046 \\
		273412 & 5678768 & 1729819 & 0 & 1682962 \\
		0 & 0 & 0 & 7+7^7 & 0 \\
		4996579 & 3315242 & 144893 & 0 & 5177634
	\end{pmatrix}.
\end{equation*}
The reciprocal characteristic polynomial of the matrix above is
\begin{gather*}
	\det(I_5-MT)=1-10823719T-36173410616147T^2+190881663422782977071T^3 \\
	-307702002432034842717713096T^4+148750558587753605666444041808300T^5.
\end{gather*}
We are interested in the $p$-adic expansion of the coefficients of this polynomial which are
\begin{align*}
	1 & =1 \\
	-10823719 & =3\cdot7^0+5\cdot7^1+6\cdot7^2+6\cdot7^3+6\cdot7^4+6\cdot7^5+5\cdot7^6+0\cdot7^7+\dots \\
	-36173410616147 & =5\cdot7^1+6\cdot7^2+6\cdot7^3+6\cdot7^4+6\cdot7^5+6\cdot7^6+0\cdot7^7+1\cdot7^8+\dots \\
	\text{coeff }T^3 & =2\cdot7^2+0\cdot7^3+0\cdot7^4+0\cdot7^5+0\cdot7^6+0\cdot7^7+1\cdot7^8+5\cdot7^9\dots \\
	\text{coeff }T^4 & =4\cdot7^3+1\cdot7^4+0\cdot7^5+0\cdot7^6+0\cdot7^7+0\cdot7^8+6\cdot7^9+3\cdot7^{10}+\dots \\
	\text{coeff }T^5 & =6\cdot7^5+6\cdot7^6+6\cdot7^7+6\cdot7^8+6\cdot7^9+6\cdot7^{10}+2\cdot7^{11}+3\cdot7^{12}+\dots \\
\end{align*}
Truncating at the seventh digit gives
\begin{align*}
	P(T) & =1+(3+5\cdot7-7^2)T+(5\cdot7-7^2)T^2+2\cdot7^2T^3+(4\cdot7^3+7^4)T^4-7^5T^5 \\
	& =1-11T-14T^2+98T^3+3773T^4-16807T^5 \\
	& =(1-7T)^3(1-(-5+2i\sqrt{6})T)(1-(-5-2i\sqrt{6})T)
\end{align*}
and therefore the zeta function of $f$ for $p=7$ is
\begin{equation*}
	\zeta(f,T)=\frac{1}{(1-T)(1-7T)(1-49T)(1-7T)^3(1+10T+49T^2)}.
\end{equation*}
This took our algorithm, which is not fully automated yet, about 3 hours.  And similar to the Cayley Cubic, a brute force point count is enough to find the zeta function of a Kummer Surface so we end this section with two examples where a direct search is not practical.

\subsection{A quartic with 6 ordinary double points}
By computer experiments the authors found the following quartic surface
\begin{equation*}
	3x_0x_1x_2(x_0+x_1)+3x_2^4-((2x_0+x_1)^2-6x_1x_2)x_3^2=0
\end{equation*}
which has 6 ordinary double points
\begin{equation*}
	[1:0:0:0],[0:1:0:0],[0:0:0:1],[1:-1:0:0],
	\Bigg[-\frac{1}{2}:1:0:-\frac{\sqrt{2}}{4}\Bigg],\Bigg[-\frac{1}{2}:1:0:\frac{\sqrt{2}}{4}\Bigg].
\end{equation*}
For this surface we can find its zeta function for $p=5,7,11,13,19$ because only the first three terms in the Frobenius action needed to be reduced in cohomology to recover $\zeta(f,T)$.  In particular, for $p=19$
\begin{align*}
	19P(T/19) & =19+30T-18T^2-77T^3-48T^4+48T^5+58T^6-12T^7 \\
	& \;\;\;\;-12T^8+58T^9+48T^{10}-48T^{11}-77T^{12}-18T^{13}+30T^{14}+19T^{15}.
\end{align*}
which is a reciprocal polynomial so for direct computation one needs to know $\#V(\mathbb{F}_{19^r})$ for $r=1,2,\dots,7$.  It takes Magma 1.172 seconds to calculate $\#V(\mathbb{F}_{19^2})=132,267$ while calculations over $\mathbb{F}_{19^2}$ would
take billions of years.  
Obviously, a better way to compute $\zeta(f,T)$ is needed.  Our algorithm is one such way and when we applied it to this surface we used the following basis for $E_2^{4,4}$
\begin{gather*}
	\{x_0^4,x_0^3x_1,x_0^3x_2,x_0^3x_3,x_0^2x_1^2,x_0^2x_1x_3,x_0^2x_2^2,x_0^2x_2x_3,x_0^2x_3^2,x_0x_1^3, \\
	x_0x_1^2x_3,x_0x_1x_2^2,x_0x_2^2x_3,x_0x_3^3\}dx_0\wedge dx_1\wedge dx_2\wedge dx_3.
\end{gather*}
The table below shows the interesting part of the zeta function for all primes between 5 and 19.  For $p=19$ it took the computer 2 days to calculate $Q_3(T)$.

\begin{table}[h!]
	\centering
	\renewcommand{\arraystretch}{1.2}
	\begin{tabular}{|c|c|}
		\hline
		$p$ & $Q_3(T)$ \\
		\hline
		\hline
		5 & $(1-5T)(1+5T)^2(1-4T+10T^2-100T^3+625T^4)(1-5^4T^4)^2$ \\
		\hline
		7 & $(1+7T)^2(1+7^3T^3)(1-8T+77T^2-392T^3+7^4T^4)(1-7^6T^6)$ \\
		\hline
		11 &  $(1-11T)^3(1-4T+66T^2-484T^3+11^4T^4)(1-11^4T^4)^2$ \\
		\hline
		13 & $(1-13T)^2(1-13^3T^3)(1+6T+143T^2+1014T^3+13^4T^4)(1-13^6T^6)$ \\
		\hline
		17 & $(1-17T)^2(1+17T)(1-36T+714 T^2-10404T^3+17^4T^4)(1-17^4T^4)^2$ \\
		\hline
		19 & $(1+19T)^2(1-19^3T^3)(1-8T-399T^2-2888T^3+19^4T^4)(1-19^6T^6)$ \\
		\hline
	\end{tabular}
	\caption{$Q_3(T)$ for the quartic with 6 nodes.}
	\label{zeta6Nodes}
\end{table}

\subsection{A quintic with 14 ordinary double points}

A quintic is of  interest because it provides us with the smallest case where $p>3$ divides the degree of $f$. Our algorithm still applies.  Let $$
f=3x_0^2x_1^2(x_0+x_1)-x_2x_3(2x_0^3+2x_1^3-x_2^3-x_3^3)\in\mathbb{Z}[x_0,x_1,x_2,x_3].
$$
Then $f$ has 14 ordinary double points listed below
\begin{align*}
	& [1:0:0:0],[1:0:0:\zeta^k\sqrt[3]{2}],[1:0:\zeta^k\sqrt[3]{2}:0], \\
	& [0:1:0:0],[0:1:0:\zeta^k\sqrt[3]{2}],[0:1:\zeta^k\sqrt[3]{2}:0]
\end{align*}
where $\zeta=e^{2\pi i/3}$ for $k=0,1,2$.  Here is he interesting part of the zeta function  for primes $p=5,7,11.$

\begin{center}
	\renewcommand{\arraystretch}{1.3}
	\begin{tabular}{ |c|c| }
		\hline
		$p$ &  $Q_3(T)$ \\
		\hline
		5 & $(1-5T)^9(1+5T)^5(1+25T^2)^2(1-6T+25T^2)$ \\
		& $(1+15T^2-100T^3+375T^4+5^6T^6)(1+15T^2+100T^3+375T^4+5^6T^6)$ \\
		& $(1+6T+45T^2+200T^3+1125T^4+3750T^5+5^6T^6)$ \\
		\hline
	   7 & $(1-7T)^{10}(1+7T)^6(1-10T+49T^2)(1-7T+49T^2)^4$ \\
		& $(1+4T+42T^2+196T^3+7^4T^4)(1+9T+91T^2+441T^3+7^4T^4)^2$ \\
		\hline
		11 & $(1-11T)^{10}(1+11T)^4(1+121T^2)^2(1-3T+121T^2)$ \\
		& $(1-T-44T^2+726T^3-5324 T^4-11^4T^5+11^6T^6)$ \\
		& $(1+T-44T^2-726T^3-5324T^4+11^4T^5+11^6T^6)$ \\
		& $(1+11T+187T^2+2178T^3+22627T^4+11^5T^5+11^6 T^6)$ \\
		\hline
	\end{tabular}
\end{center}
The dimensions of the $E_1$ and $E_2$ spectral sequence terms are as follows:
\begin{align*}
	\dim(H^3(K_f^{\bullet})_2) & =0 & \dim(H^4(K_f^{\bullet})_1) & =4 \\
	\dim(H^3(K_f^{\bullet})_7) & =10 & \dim(H^4(K_f^{\bullet})_6)= & 44 \\
	\dim(H^3(K_f^{\bullet})_{12}) & =14 & \dim(H^4(K_f^{\bullet})_{11}) & =14 \\
	\dim(E_2^{4,2}) & =4 & \dim(E_2^{4,6})=34
\end{align*}
The code 
 of the algorithm and the details of the above an some other  examples are available 
 at https://sites.google.com/view/stetson-odp-algorithm/home

\end{document}